\documentclass[11pt,reqno]{amsart}

\usepackage{amsmath, amsthm, amssymb}
\usepackage[margin=3cm]{geometry}
\usepackage{mathpazo}
\usepackage{tikz}
\usepackage{caption}
\usepackage{enumitem}
\usepackage{setspace}
\usepackage{aliascnt}
\usepackage{hyperref}
\usepackage{cleveref}

\newtheorem{thm}{Theorem}[section]
\newaliascnt{prop}{thm}
\newtheorem{prop}[prop]{Proposition}
\aliascntresetthe{prop}
\newaliascnt{lem}{thm}
\newtheorem{lem}[lem]{Lemma}
\aliascntresetthe{lem}
\newaliascnt{cor}{thm}
\newtheorem{cor}[cor]{Corollary}
\aliascntresetthe{cor}
\newaliascnt{conj}{thm}
\newtheorem{conj}[conj]{Conjecture}
\aliascntresetthe{conj}
\newaliascnt{ques}{thm}
\newtheorem{ques}[ques]{Question}
\aliascntresetthe{ques}
\theoremstyle{definition}
\newaliascnt{defn}{thm}
\newtheorem{defn}[defn]{Definition}
\aliascntresetthe{defn}
\newaliascnt{exam}{thm}
\newtheorem{exam}[exam]{Example}
\aliascntresetthe{exam}
\newaliascnt{rmk}{thm}
\newtheorem{rmk}[rmk]{Remark}
\aliascntresetthe{rmk}

\crefname{thm}{Theorem}{Theorems}
\Crefname{thm}{Theorem}{Theorems}
\crefname{prop}{Proposition}{Propositions}
\Crefname{prop}{Proposition}{Propositions}
\crefname{lem}{Lemma}{Lemmas}
\Crefname{lem}{Lemma}{Lemmas}
\crefname{cor}{Corollary}{Corollaries}
\Crefname{cor}{Corollary}{Corollaries}
\crefname{conj}{Conjecture}{Conjectures}
\Crefname{conj}{Conjecture}{Conjectures}
\crefname{ques}{Question}{Questions}
\Crefname{ques}{Question}{Questions}
\crefname{defn}{Definition}{Definitions}
\Crefname{defn}{Definition}{Definitions}
\crefname{exam}{Example}{Examples}
\Crefname{exam}{Example}{Examples}
\crefname{rmk}{Remark}{Remarks}
\Crefname{rmk}{Remark}{Remarks}
\crefname{section}{Section}{Sections}
\Crefname{section}{Section}{Sections}
\crefname{figure}{Figure}{Figures}
\Crefname{figure}{Figure}{Figures}
\crefname{table}{Table}{Tables}
\Crefname{table}{Table}{Tables}

\numberwithin{equation}{section}

\newcommand{\NN}{\mathbb{N}}
\newcommand{\ZZ}{\mathbb{Z}}
\newcommand{\xx}{\mathbf{x}}
\newcommand{\yy}{\mathbf{y}}
\newcommand{\tb}{\mathbf{t}}
\newcommand{\SSch}{\mathfrak S}
\newcommand{\id}{\mathrm{id}}
\newcommand{\length}{\ell}
\newcommand{\abs}[1]{\lvert #1\rvert}
\newcommand{\cmpl}[1]{\langle #1\rangle}
\newcommand{\qand}{\quad\mbox{and}\quad}

\title[Coefficient lifting for skew divided differences]
{Coefficient lifting for skew divided differences\\
and triple Schubert coefficients}
\author{Byung-Hak Hwang}
\address{Center for AI and Natural Sciences, Korea Institute for Advanced Study, Seoul}
\email{byunghakhwang@gmail.com}
\subjclass[2020]{Primary 05E14; Secondary 14N15, 20F55}
\keywords{Schubert polynomial, skew divided difference, Schubert structure constant, triple Schubert coefficient}

\hypersetup{hidelinks,pdftitle={Coefficient lifting for skew divided differences and triple Schubert coefficients},pdfauthor={Byung-Hak Hwang},pdfsubject={Primary 05E14; Secondary 14N15, 20F55},pdfkeywords={Schubert polynomial, skew divided difference, Schubert structure constant, triple Schubert coefficient}}

\begin{document}


\begin{abstract}
Skew divided difference operators \(\partial_{w/v}\), introduced by
Macdonald, play an important role in the study of Schubert polynomials
and their structure constants.  For permutations \(u,v,w\in S_n\), we show that
the coefficient of each monomial \(\xx^\alpha\) in
\(\partial_{w/v}\SSch_u(\xx;\yy)\), where \(\SSch_u(\xx;\yy)\) is the
double Schubert polynomial, is a single triple Schubert coefficient
evaluated at \(\tb=0\), whose indices form an explicit lift of
\((u,v,w)\) to \(S_{n+\abs{\alpha}}\).  As an application, we show that
every coefficient of \(\partial_{w/v}\SSch_u(\xx)\) is a single Schubert
structure constant.  This gives a new proof of Kirillov's conjecture
that these coefficients are nonnegative.
\end{abstract}

\maketitle


\section{Introduction}\label{sec:intro}

Schubert polynomials \( \SSch_w(\xx) \) are polynomial representatives
of Schubert classes in the cohomology of a flag variety, and play an
important role in the algebraic or combinatorial study of Schubert
calculus.  They are indexed by permutations and form a basis for the
polynomial ring \( \ZZ[\xx] \), where \( \xx = (x_1,x_2,\dots) \).  The
structure constants \( c^w_{u,v} \), defined by
\[
  \SSch_u(\xx)\SSch_v(\xx)=\sum_w c^w_{u,v}\SSch_w(\xx),
\]
can be interpreted as intersection numbers, so are manifestly
nonnegative~\cite{Macdonald1991}.

Throughout, permutations are composed as functions and act on
polynomials by permuting the variables
\(x_1,x_2,\dots\), and they fix all other variables.  For \(i\ge1\), the
\emph{divided difference} \(\partial_i\) is defined by
\begin{equation}\label{eq:dd}
\partial_if=\frac{f-s_if}{x_i-x_{i+1}},
\end{equation}
where the simple transposition \(s_i\) interchanges \(x_i\) and
\(x_{i+1}\).  Denote by \(\le\) the Bruhat order, and let \(\length(w)\)
be the length of a permutation \(w\).
For a reduced word \(w=s_{i_1}\cdots s_{i_{\length(w)}}\), the operator
\(\partial_w=\partial_{i_1}\cdots\partial_{i_{\length(w)}}\) does not depend on
the reduced word.
The skew divided difference operators were introduced by
Macdonald~\cite[Chapter~II, (2.17)]{Macdonald1991}, and
Liu~\cite[Theorem~2.4]{Liu2015} found positive expressions for them.
They are characterized by the generalized Leibniz rule
\begin{equation}\label{eq:intro-leibniz}
\partial_w(fg)=\sum_{v\le w}\bigl(\partial_{w/v}f\bigr)\bigl(\partial_vg\bigr),
\end{equation}
for all polynomials \(f\) and \(g\) (\cref{lem:leibniz}), and we set
\(\partial_{w/v}=0\) unless \(v\le w\).

Kirillov suggested the following conjecture.

\begin{conj}[{\cite[Conjecture~1]{Kirillov2007}}]\label{conj:kirillov}
For permutations \(u,v,w\), each coefficient of the polynomial
\(\partial_{w/v}\SSch_u(\xx)\) is nonnegative, that is,
\( \partial_{w/v}\SSch_u(\xx) \in \NN[\xx] \).
\end{conj}

Two earlier results on special cases of \cref{conj:kirillov} give more
than positivity.  Since \(\partial_{w/v}\) is an operator of degree
\(\length(v)-\length(w)\) and \(\SSch_u(\xx)\) is homogeneous of degree
\(\length(u)\), the polynomial
\(\partial_{w/v}\SSch_u(\xx)\) is homogeneous of degree
\(d=\length(u)+\length(v)-\length(w)\) or zero.  The first result
concerns the case \(d=0\), that is, \(\length(w)=\length(u)+\length(v)\),
in which \(\partial_{w/v}\SSch_u(\xx)\) is a constant.  This constant
is the structure constant
\(c^w_{u,v}\)~\cite[Chapter~VII]{Macdonald1991},
\cite[Proposition~3(v)]{Kirillov2007}, and this connection with the
structure constants is a main motivation for the study of skew divided
differences.

The second result concerns the case where \( v=\id \).
Taking \(g=1\) in \eqref{eq:intro-leibniz} gives
\(\partial_{w/\id}=\partial_w\), and \(\partial_w\SSch_u(\xx)\) is the
Schubert polynomial \(\SSch_{uw^{-1}}(\xx)\) or zero.
Schubert polynomials have nonnegative integer
coefficients~\cite{LS1982,Macdonald1991}, and a
combinatorial formula for these coefficients was given by Billey,
Jockusch, and Stanley~\cite{BJS1993}, with another proof
by Fomin and Stanley~\cite{FS1994}.
Moreover, Lenart, Robinson, and Sottile~\cite{LRS2006} showed that each
coefficient of a Schubert polynomial is a single structure constant.

These two results suggest that the coefficients of
\(\partial_{w/v}\SSch_u(\xx)\) have a meaning beyond their sign.  In
both cases, each coefficient is a single structure constant.  We prove
that this holds for arbitrary \(u\), \(v\), \(w\), with indices given by
an explicit construction (\cref{cor:top}).

We deduce this from a more general result on
\(\partial_{w/v}\SSch_u(\xx;\yy)\), where \(\SSch_u(\xx;\yy)\) is the
double Schubert polynomial in \(\xx\) and a second alphabet
\(\yy=(y_1,y_2,\dots)\).
Note that \(\SSch_u(\xx;0)=\SSch_u(\xx)\) (see \cref{sec:schubert}).
For a third alphabet \(\tb=(t_1,t_2,\dots)\), the \emph{triple Schubert
coefficients} \(c^w_{u,v}(\yy,\tb)\) are the coefficients in the expansion
\begin{equation}\label{eq:intro-triple}
\SSch_u(\xx;\yy)\SSch_v(\xx;\tb)=\sum_{w}c^{w}_{u,v}(\yy,\tb)\SSch_w(\xx;\tb).
\end{equation}
In the Grassmannian case, these coefficients go back to Molev and
Sagan~\cite{MS1999} and Knutson and Tao~\cite{KT2003}.  In general,
they were studied by Samuel~\cite{Samuel2024} in another convention,
and we use the convention of Gao and Xiong~\cite{GX2025}.  In the case \(\yy=\tb=0\), they
become the ordinary structure constants \(c^w_{u,v}(0,0)=c^w_{u,v}\).

For a weak composition \(\alpha=(\alpha_1,\dots,\alpha_n)\in\NN^n\), we write
\(\abs{\alpha}=\alpha_1+\cdots+\alpha_n\) and
\(\xx^\alpha=x_1^{\alpha_1}\cdots x_n^{\alpha_n}\).  For \(u,v,w\) in
the symmetric group \(S_n\) and \(\alpha\in\NN^n\), we construct in
\cref{sec:lifting} an explicit triple
\((\widetilde u_\alpha,\widetilde v_\alpha,\widetilde w_\alpha)\) of
permutations in \(S_{n+\abs{\alpha}}\), which we call the \emph{block
lift} of \((u,v,w,\alpha)\) (\cref{def:lifts}).  Roughly speaking,
\(\widetilde u_\alpha\) is \(u\) with \(\abs{\alpha}\) fixed points
appended, and \(\widetilde v_\alpha\) and \(\widetilde w_\alpha\) are
obtained from \(v\) and \(w\) by assigning to each value \(i\) a block
of \(\alpha_i+1\) consecutive values and replacing \(i\) by the
smallest element of this block in \(v\) and by the largest in \(w\).  Our main theorem is the
following.

\begin{thm}\label{thm:main}
Let \(u,v,w\in S_n\), and let
\(d=\length(u)+\length(v)-\length(w)\).  For
\(\alpha\in\NN^n\), let
\((\widetilde u_\alpha,\widetilde v_\alpha,\widetilde w_\alpha)\) be the
block lift of \((u,v,w,\alpha)\).  Then
\begin{equation}\label{eq:main}
\partial_{w/v}\SSch_u(\xx;\yy)
=\sum_{\substack{\alpha\in\NN^n\\ \abs{\alpha}\le d}}
c^{\widetilde w_\alpha}_{\widetilde u_\alpha,\widetilde v_\alpha}(\yy,0)\,\xx^\alpha.
\end{equation}
\end{thm}

In other words, the coefficient of \(\xx^\alpha\) in
\(\partial_{w/v}\SSch_u(\xx;\yy)\) is a single triple Schubert
coefficient in \(S_{n+\abs{\alpha}}\), evaluated at \(\tb=0\).  By an
identity of Samuel~\cite{Samuel2024}, we have
\(\partial_{w/v}\SSch_u(\xx;\yy)=c^w_{u,v}(\yy,\xx)\)
(\cref{prop:skew-triple}).  Thus \cref{thm:main} also says that the
coefficient of \(\tb^\alpha\) in \(c^w_{u,v}(\yy,\tb)\) is
\(c^{\widetilde w_\alpha}_{\widetilde u_\alpha,\widetilde v_\alpha}(\yy,0)\),
where \(\tb^\alpha=t_1^{\alpha_1}\cdots t_n^{\alpha_n}\).

Setting \(\yy=0\) in \cref{thm:main}, we obtain \cref{conj:kirillov}
with an explicit structure constant for each coefficient
(\cref{cor:top}).  Indeed, the coefficient of \(\xx^\alpha\) in
\(\partial_{w/v}\SSch_u(\xx;\yy)\) is homogeneous of degree
\(d-\abs{\alpha}\) in \(\yy\) or zero, so that only the terms with
\(\abs{\alpha}=d\) remain at \(\yy=0\).  For these terms the
coefficient in \eqref{eq:main} is the ordinary structure constant
\(c^{\widetilde w_\alpha}_{\widetilde u_\alpha,\widetilde v_\alpha}\), and hence
\(\partial_{w/v}\SSch_u(\xx)=\sum_{\abs{\alpha}=d}c^{\widetilde w_\alpha}_{\widetilde u_\alpha,\widetilde v_\alpha}\,\xx^\alpha\)
has nonnegative integer coefficients.  If \(d=0\), then the only term
is \(\alpha=0\), whose block lift is \((u,v,w)\) itself, and we recover
the first result above (\cref{cor:zero}).  Since
\(\partial_w\SSch_u(\xx)\) is a Schubert polynomial or zero, the second
result amounts to the case \(v=w=\id\), where
\(\partial_{w/v}\SSch_u(\xx)=\SSch_u(\xx)\).  In this case
\cref{cor:vw} realizes each coefficient as a single structure constant,
as Lenart, Robinson, and Sottile did, and we compare the two
realizations in \cref{sec:lrs}.

We should note that \cref{conj:kirillov} has already been proved by
Gao and Xiong~\cite[Corollary~1.2]{GX2025}.  They proved a conjecture
of Samuel~\cite[Conjecture~1.1]{Samuel2024} that
\(c^w_{u,v}(\yy,\tb)\) is a polynomial in the differences \(t_i-y_j\)
with nonnegative integer coefficients~\cite[Theorem~1.1]{GX2025}, and
deduced Kirillov's conjecture from it.  Their proof relies on a refined
version of Graham's positivity theorem~\cite[Theorem~3.2]{Graham2001}.
In contrast, our proof of \cref{thm:main} uses only divided
differences, and \cref{cor:top} follows from it and the nonnegativity
of the structure constants.  It would be interesting to find a direct
connection between the geometric approach of Gao and Xiong and our
algebro-combinatorial one.

This paper is organized as follows.  In \cref{sec:conventions} we
recall the background and prove the identity
\(\partial_{w/v}\SSch_u(\xx;\yy)=c^w_{u,v}(\yy,\xx)\)
(\cref{prop:skew-triple}).  In \cref{sec:extraction} we introduce the
value blocks of \(\alpha\) and extract the coefficient of \(\xx^\alpha\)
in a polynomial by divided differences followed by evaluation at zero
(\cref{prop:extraction}).  In \cref{sec:lifting} we prove a
factorization rule for skew divided differences (\cref{lem:factor}),
define the block lift, and show that composing \(\partial_{w/v}\) with
the extraction operator of \cref{sec:extraction} gives
\(\partial_{\widetilde w_\alpha/\widetilde v_\alpha}\)
(\cref{thm:operator}).  In \cref{sec:main} we prove \cref{thm:main},
derive \cref{conj:kirillov} (\cref{cor:top}) and further consequences,
pose a question (\cref{ques:signed}), compare our realization of the
coefficients of Schubert polynomials with that of Lenart, Robinson, and
Sottile, and complete the running example.


\section{Conventions and the objects in the statement}
\label{sec:conventions}

None of the results in this section are new, but for completeness we
provide short proofs for several of them.  For background on Schubert
polynomials and on permutations, we refer to~\cite{Macdonald1991,BB2005}.

\subsection{Permutations}

For \(k\ge1\) we write \([k]=\{1,2,\dots,k\}\), and set \([0]=\varnothing\).
We write a permutation \(w\in S_N\) in one-line notation
\(w=w(1)w(2)\cdots w(N)\), and we call \(i\) a position and \(w(i)\)
the value at that position.
We denote by \(s_i\) the simple transposition that interchanges \(i\)
and \(i+1\).
As in the introduction, permutations are
composed as functions, that is,
\begin{equation}\label{eq:composition}
(vw)(i)=v(w(i)).
\end{equation}
Recall that if
\(w=s_{i_1}\cdots s_{i_{\length(w)}}\) is a reduced word, then
\begin{equation}\label{eq:bruhat}
v\le w\iff
\text{some subword \(s_{i_{j_1}}\cdots s_{i_{j_m}}\)
with \(j_1<\cdots<j_m\) is a reduced word for \(v\)}.
\end{equation}
We say that \(s_i\) is a \emph{left ascent} of \(w\) if
\(\length(s_iw)=\length(w)+1\), and a \emph{left descent} of \(w\) if
\(\length(s_iw)=\length(w)-1\).  Exactly one of the two holds, and
\begin{equation}\label{eq:ascent}
s_i\text{ is a left ascent of }w
\iff
w^{-1}(i)<w^{-1}(i+1).
\end{equation}

For \(n\le N\), we regard \(S_n\) as the subgroup of \(S_N\) that fixes
\(n+1,\dots,N\), and we write \(w\times1_{N-n}\) for \(w\in S_n\)
regarded as an element of \(S_N\).
Let \(\sigma\in S_N\) be increasing on \([n]\), that is,
\(\sigma(1)<\sigma(2)<\cdots<\sigma(n)\).  Then
\(\length(\sigma s_i)>\length(\sigma)\) for every \(i<n\), and by
\cite[Proposition~2.4.4]{BB2005}, we have
\begin{equation}\label{eq:prefix-length}
\length(\sigma w)=\length(\sigma)+\length(w)
\qquad(w\in S_n).
\end{equation}
In particular, \(\sigma\) is the unique element of smallest length in
the coset \(\sigma S_n\).

\begin{defn}\label{def:completion}
Let \(n\le N\), and let \(m_1,\dots,m_n\) be distinct elements of
\([N]\).  The \emph{increasing completion}
\(\cmpl{m_1,\dots,m_n}\in S_N\) is the permutation whose one-line
notation begins with \(m_1,\dots,m_n\) and continues with the elements
of \([N]\setminus\{m_1,\dots,m_n\}\) in increasing order.
\end{defn}

For example, \(w\times1_{N-n}=\cmpl{w(1),\dots,w(n)}\) for
\(w\in S_n\).  In examples with \(n<N\), a vertical bar separates the first \(n\)
entries of the one-line notation from the others.
Since the last \(N-n\) entries of \(\cmpl{m_1,\dots,m_n}\)
increase, each of its inversions pairs a position \(i\le n\) either
with a position \(j\) such that \(i<j\le n\) and \(m_i>m_j\) or with the
position of an appended value smaller than \(m_i\), so that
\begin{equation}\label{eq:completion-length}
\length\bigl(\cmpl{m_1,\dots,m_n}\bigr)
=\#\{i<j\le n: m_i>m_j\}
+\sum_{i=1}^n\#\{j\in[N]\setminus\{m_1,\dots,m_n\}: j<m_i\}.
\end{equation}
For \(w\in S_n\), it is straightforward to check that
\begin{equation}\label{eq:completion-compose}
\cmpl{m_1,\dots,m_n}\,(w\times1_{N-n})=\cmpl{m_{w(1)},\dots,m_{w(n)}}.
\end{equation}
If \(2\le k\le N\) and \(k\notin\{m_1,\dots,m_n\}\), then
\begin{equation}\label{eq:completion-ascent}
s_{k-1}\text{ is a left ascent of }\cmpl{m_1,\dots,m_n},
\end{equation}
because the value \(k\) stands at a position after \(n\), and the value
\(k-1\) stands either at a position at most \(n\) or at an earlier
position among the increasing last \(N-n\) entries, so that
\eqref{eq:ascent} applies.

\begin{exam}\label{ex:data}
Throughout the paper, our running example is \(n=6\) and
\[
u=435126,\qquad v=142635,\qquad w=164325.
\]
We have \(\length(u)=7\), \(\length(v)=4\), and
\(\length(w)=7\), so that \(d=\length(u)+\length(v)-\length(w)=4\).
We fix the reduced word \(w=s_2s_3s_2s_5s_4s_3s_2\).  One can check
that the two underlined subwords in
\[
  s_2\underline{s_3}\underline{s_2}\underline{s_5}\underline{s_4}s_3s_2 \qand
  s_2\underline{s_3}s_2\underline{s_5}\underline{s_4}s_3\underline{s_2}
\]
are reduced words for \(v=s_3s_2s_5s_4=s_3s_5s_4s_2\), and that no
other subword is.  Thus \(v\le w\) by \eqref{eq:bruhat}.
Since the value \(3\) stands to the left of the value
\(2\) in \(w\), the transposition \(s_2\) is a left descent of \(w\) by
\eqref{eq:ascent}, and indeed the reduced word for \(w\) above begins
with \(s_2\).  In the notation of \cref{def:completion},
\(v\times1_3=\cmpl{1,4,2,6,3,5}=1\,4\,2\,6\,3\,5\mid7\,8\,9\) in
\(S_9\), and \eqref{eq:completion-length} gives
\(\length(v\times1_3)=4+0=\length(v)\).  By
\eqref{eq:completion-ascent}, \(s_6\), \(s_7\), and \(s_8\) are left
ascents of \(v\times1_3\), since \(7\), \(8\), and \(9\) are not among
its first six entries.
\end{exam}

\subsection{Polynomials and divided differences}\label{sec:dd}

Throughout, \(A\) denotes a commutative ring, in practice \(\ZZ\),
\(\ZZ[\yy]\), or \(\ZZ[\yy,\tb]\) for the alphabets \(\yy\) and \(\tb\),
and \(A[\xx]\) denotes the polynomial ring over
\(A\) in \(\xx=(x_1,x_2,\dots)\).  The group \(S_N\) acts on \(A[\xx]\)
by permuting the variables and fixing \(A\), that is,
\begin{equation}\label{eq:action}
(wf)(x_1,x_2,\dots)=f(x_{w(1)},x_{w(2)},\dots),
\end{equation}
where \(w(i)=i\) for \(i>N\).  By \eqref{eq:composition} this is a
left action, \((vw)f=v(wf)\).
The divided differences \(\partial_i\), introduced by Bernstein,
Gelfand, and Gelfand~\cite{BGG1973} and by
Demazure~\cite{Demazure1974}, are given by \eqref{eq:dd}, and
\(\partial_w\) is as in \cref{sec:intro}.  A product of operators acts
with its rightmost factor first.  A product \(\partial_{i_1}\cdots\partial_{i_m}\) is zero
if the word \(s_{i_1}\cdots s_{i_m}\) is not reduced,
and for \(v,w\in S_N\) we have \(\partial_v\partial_w=\partial_{vw}\)
if \(\length(vw)=\length(v)+\length(w)\) and
\(\partial_v\partial_w=0\)
otherwise~\cite[Chapter~II, (2.6) and (2.7)]{Macdonald1991}.  Each
\(\partial_i\) lowers degrees by one, and it is linear over the
subring of \(A[\xx]\) in which \(x_i\) and \(x_{i+1}\)
do not occur.  A polynomial in \(x_1,\dots,x_n\) is also regarded as a
polynomial in \(x_1,\dots,x_N\) for \(N>n\), in which
\(x_{n+1},\dots,x_N\) do not occur.

\subsection{Double Schubert polynomials}\label{sec:schubert}

Following Macdonald~\cite[Chapter~VI, (6.1) and (6.2)]{Macdonald1991},
we define the \emph{double Schubert polynomial} of \(w\in S_n\),
introduced by Lascoux~\cite{Lascoux1982}, by
\[
\SSch_w(\xx;\yy)=\partial_{w^{-1}w_0}\prod_{i+j\le n}(x_i-y_j),
\]
where \(w_0=n\,(n-1)\cdots1\) is the longest element of \(S_n\), and
the \emph{Schubert polynomial} of \(w\), introduced by Lascoux and
Sch\"utzenberger~\cite{LS1982}, by \(\SSch_w(\xx)=\SSch_w(\xx;0)\).  We collect the properties of
double Schubert polynomials that we use.
\begin{enumerate}[label=(S\arabic*),leftmargin=3em]
\item\label{S:hom}
(Homogeneity) For \(w\in S_n\), the polynomial
\(\SSch_w(\xx;\yy)\) is homogeneous of degree \(\length(w)\) in
\(x_1,\dots,x_{n-1},y_1,\dots,y_{n-1}\), and
\(\SSch_{\id}(\xx;\yy)=1\).
\item\label{S:stable}
(Stability) \(\SSch_{w\times1_m}(\xx;\yy)=\SSch_w(\xx;\yy)\) for
\(w\in S_n\) and \(m\ge0\).
Thus \(\SSch_w(\xx;\yy)\) is defined for \(w\) in
\(S_\infty=\bigcup_{n\ge1}S_n\).
\item\label{S:word}
(Word action) For \(v,w\in S_\infty\), we have
\[
  \partial_v\SSch_w(\xx;\yy) = \begin{cases}
    \SSch_{wv^{-1}}(\xx;\yy) & \mbox{if \( \length(wv^{-1})=\length(w)-\length(v) \),} \\ 
    0 & \mbox{otherwise.}
  \end{cases}
\]
\item\label{S:diag}
(Diagonal vanishing) For \(w\in S_\infty\), we have
\[
  \SSch_w(\yy;\yy) = \begin{cases}
    1 & \mbox{if \( w=\id \),} \\ 
    0 & \mbox{otherwise.}
  \end{cases}
\]
\item\label{S:basis}
(Basis) Every \(G\in A[\yy][\xx]\) has a unique expansion
\(G=\sum_w\kappa_w(\yy)\,\SSch_w(\xx;\yy)\) over
\(w\in S_\infty\) with finitely many nonzero coefficients
\(\kappa_w(\yy)\in A[\yy]\).
\end{enumerate}
These properties can be found in or follow directly
from~\cite[Chapter~VI]{Macdonald1991}.
By \ref{S:word} and \ref{S:diag}, the coefficients of the expansion in
\ref{S:basis} are given by
\begin{equation}\label{eq:functional}
\kappa_w(\yy)=\bigl(\partial_w G\bigr)\big|_{\xx=\yy}.
\end{equation}
The same holds with \(\yy\) replaced by the alphabet \(\tb\) or by
\(0\).

\subsection{Skew divided differences}\label{sec:skew}

In \cref{sec:intro}, the skew divided differences were determined by
the generalized Leibniz rule \eqref{eq:intro-leibniz}.  In this
subsection, we take a subword formula as their definition and prove
that the operators it defines satisfy \eqref{eq:intro-leibniz} and are
the only ones that do (\cref{lem:leibniz}).

\begin{defn}\label{def:skew}
Let \(v,w\in S_N\), and fix a reduced word
\(w=s_{i_1}\cdots s_{i_{\length(w)}}\).  The \emph{skew divided
difference} \(\partial_{w/v}\) is the operator on \(A[\xx]\) given
by
\begin{equation}\label{eq:subword}
\partial_{w/v}
=\sum_{J}\phi_1^{J}\phi_2^{J}\cdots\phi_{\length(w)}^{J},
\end{equation}
where the sum is over the subsets \(J\subseteq[\length(w)]\) such
that the letters \(s_{i_j}\) with \(j\in J\), in increasing order of
\(j\), form a reduced word for \(v\), and where
\[
  \phi_j^{J} = \begin{cases}
    s_{i_j} & \mbox{if \( j\in J \),} \\ 
    \partial_{i_j} & \mbox{otherwise.}
  \end{cases}
\]
We call \(v\) the
\emph{lower endpoint} and \(w\) the \emph{upper endpoint} of
\(\partial_{w/v}\).
\end{defn}

The following properties are immediate from \cref{def:skew}.
\begin{enumerate}[label=(D\arabic*),leftmargin=3em]
\item\label{it:skew-bruhat}
\(\partial_{w/v}=0\) unless \(v\le w\).
\item\label{it:skew-ends}
\(\partial_{w/\id}=\partial_w\), since only \(J=\varnothing\)
contributes, and \(\partial_{w/w}=w\), since only
\(J=[\length(w)]\) contributes.
\item\label{it:skew-degree}
\(\partial_{w/v}\) is \(A\)-linear, lowers degrees by
\(\length(w)-\length(v)\), and commutes with every ring
homomorphism applied to the coefficients, for example with setting
\(\yy=0\) when \(A=\ZZ[\yy]\).
\end{enumerate}
A reduced word for \(w\in S_N\) is also a reduced word for
\(w\times1_m\in S_{N+m}\), with the same contributing subsets \(J\),
so \eqref{eq:subword} does not depend on \(N\), and
\(\partial_{w/v}\) is defined for all \(v,w\in S_\infty\).
The following lemma is due to Macdonald, who states it in another
normalization (see \cref{rmk:macdonald}).  We include its short proof
in our conventions.

\begin{lem}[{\cite[Chapter~II, (2.17)]{Macdonald1991}}]\label{lem:leibniz}
For \(w\in S_N\) and \(f,g\in A[\xx]\), we have
\begin{equation}\label{eq:leibniz}
\partial_w(fg)
=\sum_{v\in S_N}\bigl(\partial_{w/v}f\bigr)\bigl(\partial_v g\bigr).
\end{equation}
Conversely, the operators \(\partial_{w/v}\) \((v\in S_N)\) are the
only operators on \(A[\xx]\) that satisfy \eqref{eq:leibniz} for all
\(f,g\in A[\xx]\).  In particular, \(\partial_{w/v}\) does not depend on the choice of the reduced
word for \(w\).
\end{lem}

\begin{proof}
We first prove \eqref{eq:leibniz}.  By \eqref{eq:dd}, we have
\(\partial_i(fg)=(\partial_if)\,g+(s_if)\,(\partial_ig)\).  We apply
this rule for the letters \(s_{i_{\length(w)}},\dots,s_{i_1}\) of
the reduced word in turn, starting from the right.  At the letter
\(s_{i_j}\), either \(\partial_{i_j}\) acts on the first factor, or
\(s_{i_j}\) acts on the first factor and \(\partial_{i_j}\) on the
second.  Recording the second choice by \(j\in J\), we obtain
\[
\partial_w(fg)
=\sum_{J\subseteq[\length(w)]}
\bigl(\phi^{J}_{1}\cdots\phi^{J}_{\length(w)}f\bigr)
\bigl(\partial_{i_{j_1}}\cdots\partial_{i_{j_m}}g\bigr),
\]
where \(J=\{j_1<\cdots<j_m\}\).  The second factor
is \(\partial_v g\) if \(s_{i_{j_1}}\cdots s_{i_{j_m}}\) is a reduced
word for some \(v\in S_N\), and it is zero otherwise.  Grouping the
terms according to \(v\), we get \eqref{eq:leibniz}.

It remains to prove the uniqueness.  Let \(T_v\) \((v\in S_N)\) be
operators on \(A[\xx]\) that satisfy \eqref{eq:leibniz} with \(T_v\) in
place of \(\partial_{w/v}\), put \(D_vf=T_vf-\partial_{w/v}f\), and fix
\(f\in A[\xx]\).  Subtracting \eqref{eq:leibniz} from
the rule satisfied by the operators \(T_v\) gives
\[
\sum_{v\in S_N}\bigl(D_vf\bigr)\bigl(\partial_v g\bigr)=0
\qquad\text{for all }g\in A[\xx].
\]
For a contradiction, suppose that \(D_vf\ne0\) for some \(v\), and
choose \(u\in S_N\) with \(D_uf\ne0\) and \(\length(u)\)
as small as possible.  We take \(g=\SSch_u(\xx)\).  Let
\(v\ne u\) with \(D_vf\ne0\).  Then
\(\length(v)\ge\length(u)\) by the choice of \(u\), so the
equality \(\length(uv^{-1})=\length(u)-\length(v)\) would
force \(uv^{-1}=\id\), which is impossible, and
\(\partial_v\SSch_u(\xx)=0\) by \ref{S:word} with \(\yy=0\).
Thus only the term with \(v=u\) remains, and since
\(\partial_u\SSch_u(\xx)=\SSch_{\id}(\xx)=1\) by \ref{S:word}
and \ref{S:hom}, the last display becomes \(D_uf=0\), a
contradiction.  It follows that \(T_v=\partial_{w/v}\) for every
\(v\).  Since the first part holds for every reduced word for \(w\),
the operators given by \eqref{eq:subword} for two reduced words both
satisfy \eqref{eq:leibniz} and therefore coincide, which completes the
proof.
\end{proof}

\begin{rmk}\label{rmk:macdonald}
  Our operators \( \partial_{w/v} \) are those of Gao and
  Xiong~\cite[Definition~2.6]{GX2025} and differ by a permutation from
  those of Macdonald~\cite[Chapter~II, (2.17)]{Macdonald1991} and
  Kirillov~\cite[Definition~4 and (5.1)]{Kirillov2007}, whose operator
  \(\partial^{\mathrm M}_{w/v}\) satisfies
  \[
  \partial_w(fg)=\sum_{v}v\bigl(\partial^{\mathrm M}_{w/v}f\bigr)\bigl(\partial_v g\bigr),
  \]
  so that
  \[
  \partial_{w/v}=v\circ\partial^{\mathrm M}_{w/v}.
  \]
  Since \(v\) only permutes the variables, the polynomials
  \(\partial_{w/v}f\) and \(\partial^{\mathrm M}_{w/v}f\) have
  the same coefficients up to a relabeling of the monomials, so
  \(\partial_{w/v}f\) has nonnegative coefficients if and only if
  \(\partial^{\mathrm M}_{w/v}f\) does.
\end{rmk}

\begin{exam}\label{ex:skew}
Take the same setting as in \cref{ex:data}.  By \cref{def:skew}, the
two underlined subwords there give
\begin{equation}\label{eq:ex-operator}
\partial_{w/v}
=\partial_2\,s_3s_2s_5s_4\,\partial_3\partial_2
+\partial_2\,s_3\,\partial_2\,s_5s_4\,\partial_3\,s_2.
\end{equation}
\end{exam}

\subsection{Triple Schubert coefficients}\label{sec:triple}

Applying \ref{S:basis} with the coefficient ring \(\ZZ[\yy]\) and
with \(\tb\) in place of \(\yy\), we obtain unique polynomials
\(c^w_{u,v}(\yy,\tb)\in\ZZ[\yy,\tb]\) for \(u,v,w\in S_\infty\) such
that \eqref{eq:intro-triple} holds with finitely many nonzero terms.
Since both sides of \eqref{eq:intro-triple} are homogeneous by
\ref{S:hom} and the expansion is unique, \(c^w_{u,v}(\yy,\tb)\) is
homogeneous of degree \(\length(u)+\length(v)-\length(w)\) or zero.
Setting \(\tb=0\) in \eqref{eq:intro-triple} expands
\(\SSch_u(\xx;\yy)\SSch_v(\xx)\) in Schubert polynomials with the
coefficients \(c^w_{u,v}(\yy,0)\), and setting \(\yy=0\) as well shows
that \(c^w_{u,v}(0,0)=c^w_{u,v}\) is the ordinary structure constant.
Both expansions are unique by \ref{S:basis} with
\(0\) in place of \(\yy\).

The following identity is due to
Samuel~\cite{Samuel2024}, who derives it in another
notation.  For completeness, we include its short proof in our
conventions.

\begin{prop}[{\cite{Samuel2024}}]\label{prop:skew-triple}
For \(u,v,w\in S_\infty\), we have
\begin{equation}\label{eq:skew-triple}
\partial_{w/v}\SSch_u(\xx;\yy)=c^{w}_{u,v}(\yy,\xx).
\end{equation}
Moreover, \(c^{w}_{u,v}(\yy,\tb)=0\) unless \(v\le w\), and if
\(u,v,w\in S_n\), then both sides of \eqref{eq:skew-triple} are
polynomials in \(x_1,\dots,x_n\) with coefficients in \(\ZZ[\yy]\).
\end{prop}

\begin{proof}
By \eqref{eq:functional} in the alphabet \(\tb\), the coefficient of
\(\SSch_w(\xx;\tb)\) in \eqref{eq:intro-triple} is
\[
c^{w}_{u,v}(\yy,\tb)
=\Bigl(\partial_{w}\bigl(\SSch_u(\xx;\yy)\SSch_v(\xx;\tb)\bigr)\Bigr)\Big|_{\xx=\tb}.
\]
Expanding \(\partial_w\) of the product by \eqref{eq:leibniz} with
\(A=\ZZ[\yy,\tb]\) gives
\[
c^{w}_{u,v}(\yy,\tb)
=\sum_{z}\Bigl(\bigl(\partial_{w/z}\SSch_u(\xx;\yy)\bigr)\,
\bigl(\partial_{z}\SSch_v(\xx;\tb)\bigr)\Bigr)\Big|_{\xx=\tb}.
\]
By \ref{S:word} and \ref{S:diag}, the value of
\(\partial_{z}\SSch_v(\xx;\tb)\) at \(\xx=\tb\) is \(1\) if
\(z=v\) and \(0\) otherwise, so only the term with \(z=v\) remains,
and
\[
c^{w}_{u,v}(\yy,\tb)=\bigl(\partial_{w/v}\SSch_u(\xx;\yy)\bigr)\big|_{\xx=\tb}.
\]
Renaming \(\tb\) as \(\xx\) gives \eqref{eq:skew-triple}, and the
vanishing for \(v\not\le w\) follows from \ref{it:skew-bruhat}.
Finally, if \(u,v,w\in S_n\), then \(\partial_{w/v}\) involves only
\(s_i\) and \(\partial_i\) with \(i<n\), and \(\SSch_u(\xx;\yy)\)
involves only \(x_i\) with \(i<n\) by \ref{S:hom}, so
\(\partial_{w/v}\SSch_u(\xx;\yy)\) is a polynomial in
\(x_1,\dots,x_n\).
\end{proof}

\begin{rmk}\label{rmk:fgx}
Fan, Guo, and Xiong~\cite[Proposition~6.4]{FGX2025} proved the case
\(\yy=0\) of \eqref{eq:skew-triple} and used it to prove
\cref{conj:kirillov} when \(u\) and \(v\) have separated
descents~\cite[Corollary~6.6]{FGX2025}, that is, when \(i\le j\) for
all positions \(i\) and \(j\) with \(u(i)>u(i+1)\) and
\(v(j)>v(j+1)\).  This case also follows from an earlier formula of
Samuel~\cite[Theorem~3.1]{Samuel2024}, as Fan, Guo, and Xiong
note~\cite[Remark~2.4]{FGX2025}.
\end{rmk}

\begin{exam}\label{ex:F}
Take the same setting as in \cref{ex:skew}.  Since \(u\) is
dominant~\cite[Chapter~I, (1.30)]{Macdonald1991}, Macdonald's product
formula for dominant permutations~\cite[Chapter~VI, (6.14)]{Macdonald1991}
gives
\[
\SSch_u(\xx;\yy)
=(x_1-y_1)(x_1-y_2)(x_1-y_3)(x_2-y_1)(x_2-y_2)(x_3-y_1)(x_3-y_2).
\]
Since none of the operators in
\eqref{eq:ex-operator} involves \(s_1\) or \(\partial_1\), they
treat the first three factors, which involve only \(x_1\) and \(\yy\),
as constants.  Applying \eqref{eq:ex-operator} to the product of the
other four factors, we obtain
\begin{equation}\label{eq:ex-F}
F:=\partial_{w/v}\SSch_u(\xx;\yy)
=(x_1-y_1)(x_1-y_2)(x_1-y_3)\,(x_4+x_6-y_1-y_2),
\end{equation}
which is homogeneous of degree \(d=4\).  We write
\(e_1=y_1+y_2+y_3\), \(e_2=y_1y_2+y_1y_3+y_2y_3\), and
\(e_3=y_1y_2y_3\), so that
\(F=(x_1^3-e_1x_1^2+e_2x_1-e_3)(x_4+x_6-y_1-y_2)\).  The expansion of
\(F\) in monomials in \(\xx\) is then
\begin{equation}\label{eq:ex-coefficients}
\begin{split}
F&=x_1^3x_4+x_1^3x_6-(y_1+y_2)\,x_1^3
-e_1x_1^2x_4-e_1x_1^2x_6+e_1(y_1+y_2)\,x_1^2\\
&\quad+e_2x_1x_4+e_2x_1x_6-e_2(y_1+y_2)\,x_1
-e_3x_4-e_3x_6+e_3(y_1+y_2),
\end{split}
\end{equation}
with twelve nonzero coefficients.  Since \(F\) is homogeneous of degree
\(d\), its part of degree \(d\) in \(\xx\) is its value at \(\yy=0\),
which is \(\partial_{w/v}\SSch_u(\xx)=x_1^3x_4+x_1^3x_6\) by
\ref{it:skew-degree}.  Its coefficients are nonnegative, as
\cref{conj:kirillov} predicts.  By
\cref{prop:skew-triple}, we have
\begin{equation}\label{eq:ex-triple}
c^{164325}_{435126,\,142635}(\yy,\tb)
=(t_1-y_1)(t_1-y_2)(t_1-y_3)\,(t_4+t_6-y_1-y_2),
\end{equation}
and the coefficient of \(t_1^2t_4\) in this triple Schubert coefficient
is the coefficient \(-e_1=-y_1-y_2-y_3\) of \(x_1^2x_4\) in \(F\).
\end{exam}


\section{Extracting a coefficient with divided differences}
\label{sec:extraction}

Our goal in this section is to prove \cref{prop:extraction}, which
expresses the coefficient of \(\xx^\alpha\) in a polynomial in terms of
divided differences and an evaluation at zero.  The basic idea is to
rename the variables so that each \(x_i\) is followed by \(\alpha_i\)
variables that do not occur, and then to extract the coefficient of
\(x_i^{\alpha_i}\) by a chain of \(\alpha_i\) divided differences
(\cref{lem:newton}).

For \(f\in A[\xx]\) and \(\alpha\in\NN^n\), we denote
by \([\xx^\alpha]f\) the coefficient of
\(\xx^\alpha=x_1^{\alpha_1}\cdots x_n^{\alpha_n}\) in \(f\), regarded
as a polynomial in \(x_1,\dots,x_n\) whose coefficients are polynomials
in the remaining variables.

\subsection{One variable}

For \(i\le j\) and \(m\ge0\), we use \(h_m(x_i,\dots,x_j)\) to denote
the complete homogeneous symmetric polynomial of degree \(m\) in
\(x_i,\dots,x_j\), the sum of all monomials of degree \(m\) in these
variables.  We set \(h_0=1\) and define \(h_m\) to be zero for
\(m<0\).
For an indeterminate \(T\), these polynomials have the generating
function
\begin{equation}\label{eq:h-gf}
\sum_{m\ge0}h_m(x_i,\dots,x_{j})\,T^m
=\frac{1}{(1-x_iT)(1-x_{i+1}T)\cdots(1-x_{j}T)}.
\end{equation}

\begin{lem}\label{lem:newton}
Let \(i\ge1\) and \(k\ge0\), with the convention that
\(\partial_{i+k-1}\cdots\partial_{i+1}\partial_i\) is the identity
operator if \(k=0\).  Then the following hold.
\begin{enumerate}[label=(\roman*),leftmargin=2.2em]
\item\label{it:newton-chain}
For \(m\ge0\), we have
\begin{equation}\label{eq:newton}
\partial_{i+k-1}\cdots\partial_{i+1}\partial_i\bigl(x_i^{m}\bigr)
=h_{m-k}(x_i,x_{i+1},\dots,x_{i+k}).
\end{equation}
\item\label{it:newton-coeff}
If \(f\in A[\xx]\) is a polynomial in which \(x_{i+1},\dots,x_{i+k}\)
do not occur, then
\[
\bigl(\partial_{i+k-1}\cdots\partial_{i+1}\partial_if\bigr)
\big|_{x_i=x_{i+1}=\cdots=x_{i+k}=0}
\]
is the coefficient of \(x_i^k\) in \(f\), regarded as a polynomial in
\(x_i\).
\end{enumerate}
\end{lem}

\begin{proof}
We prove \ref{it:newton-chain} by a generating function argument,
letting the divided differences act coefficientwise on power series in
\(T\).  For \(j\ge0\), the operator \(\partial_{i+j}\) is then linear
over the power series whose coefficients do not involve \(x_{i+j}\) and
\(x_{i+j+1}\), and \eqref{eq:dd} gives
\[
\partial_{i+j}\frac{1}{1-x_{i+j}T}
=\frac{1}{x_{i+j}-x_{i+j+1}}
\Bigl(\frac{1}{1-x_{i+j}T}-\frac{1}{1-x_{i+j+1}T}\Bigr)
=\frac{T}{(1-x_{i+j}T)(1-x_{i+j+1}T)}.
\]
Applying this identity successively for \(j=0,1,\dots,k-1\), we obtain
\begin{align*}
  \partial_{i+k-1}\cdots\partial_i \frac{1}{1-x_i T}
    &= \frac{T^k}{\prod_{j=0}^k(1-x_{i+j}T)}
     =\sum_{m\ge0}h_{m-k}(x_i,\dots,x_{i+k})T^m,
\end{align*}
where the operator is the identity if \(k=0\), and the last equality
is \eqref{eq:h-gf}.  Comparing the
coefficients of \(T^m\) on both sides proves \eqref{eq:newton}.

For \ref{it:newton-coeff}, write \(f=\sum_{m\ge0}f_m\,x_i^m\), where
the polynomials \(f_m\) do not involve \(x_i,\dots,x_{i+k}\).  Then
\eqref{eq:newton} gives
\[
\left(\partial_{i+k-1}\cdots\partial_if\right)
\big|_{x_i=\cdots=x_{i+k}=0}
=\sum_{m\ge0}f_m\,h_{m-k}(x_i,\dots,x_{i+k})
\big|_{x_i=\cdots=x_{i+k}=0}.
\]
At zero, \(h_{m-k}\) equals \(1\) if \(m=k\) and \(0\) otherwise.  Thus
the right side is \(f_k\), which is the coefficient of \(x_i^k\) in
\(f\), as desired.
\end{proof}

For \(i=1\), \cref{lem:newton}\ref{it:newton-chain} follows by
applying the first identity of~\cite[Chapter~III, (3.9)]{Macdonald1991}
\(k\) times, and \cref{lem:newton}\ref{it:newton-coeff} follows by
setting \(y_1=y_2=\cdots=0\) in Newton's interpolation formula, a
special case of~\cite[Chapter~VI, (6.8)]{Macdonald1991}.  We
remark that \cref{lem:newton}\ref{it:newton-coeff} puts no
condition on the variables other than \(x_{i+1},\dots,x_{i+k}\).  The
polynomial \(f\) may involve \(x_1,\dots,x_{i-1}\) and \(x_{i+k+1}\),
\(x_{i+k+2}\), \dots, and the coefficient of \(x_i^k\) is then a
polynomial in them.

\subsection{Several variables}

In this subsection, we fix \(\alpha\in\NN^n\) and put
\(r=\abs{\alpha}\) and \(N=n+r\).

\begin{defn}\label{def:blocks}
For \(i\in[n]\), the \emph{base value} \(b_i\) and the \emph{value block} \(B_i\) are
\begin{equation}\label{eq:blocks}
b_i=i+\alpha_1+\cdots+\alpha_{i-1} \qand
B_i=\{b_i,b_i+1,\dots,b_i+\alpha_i\}.
\end{equation}
The \emph{block placement} is
\begin{equation}\label{eq:sigma}
\sigma_\alpha=\cmpl{b_1,b_2,\dots,b_n}\in S_N,
\end{equation}
and the \emph{block cycle} \(\rho_i\) of \(i\in[n]\) and their product
\(\rho_\alpha\) are
\begin{equation}\label{eq:rho}
\rho_i=s_{b_i+\alpha_i-1}\cdots s_{b_i+1}s_{b_i}\in S_N
\qand
\rho_\alpha=\rho_1\rho_2\cdots\rho_n.
\end{equation}
Here, we set \( \rho_i=\id \) if \( \alpha_i=0 \).
\end{defn}

The base value \(b_i\) is the smallest element of \(B_i\), and
\(b_i+\alpha_i\) is the largest, which we call the \emph{top} of \(B_i\).
The main reason for introducing the block placement is the renaming
\eqref{eq:sigma-substitution} in \cref{lem:blocks} below.  It replaces
each \(x_i\) by \(x_{b_i}\), so that the \(\alpha_i\) variables
\(x_{b_i+1},\dots,x_{b_i+\alpha_i}\) do not occur, as
\cref{lem:newton}\ref{it:newton-coeff} requires for the variable
\(x_{b_i}\) and \(k=\alpha_i\).  The block cycle \(\rho_i\) is chosen so that
\(\partial_{\rho_i}\) is the chain
\(\partial_{b_i+\alpha_i-1}\cdots\partial_{b_i}\) of \cref{lem:newton}
(\cref{lem:blocks}\ref{it:reduced}).

\begin{exam}\label{ex:blocks}
Take the same setting as in \cref{ex:F}, and let
\(\alpha=(2,0,0,1,0,0)\), so that \(\xx^\alpha=x_1^2x_4\), \(r=3\),
and \(N=9\).  \Cref{def:blocks} gives the base values and
the value blocks in the table
\[
\begin{array}{c|cccccc}
i&1&2&3&4&5&6\\\hline
b_i&1&4&5&6&8&9\\
B_i&\{1,2,3\}&\{4\}&\{5\}&\{6,7\}&\{8\}&\{9\}
\end{array},
\]
and \cref{fig:blocks} shows the blocks as consecutive intervals of
\([9]\).  The block placement is
\[
\sigma_\alpha=\cmpl{1,4,5,6,8,9}=1\,4\,5\,6\,8\,9\mid2\,3\,7,
\]
and the block cycles are \(\rho_1=s_2s_1\), \(\rho_4=s_6\), and
\(\rho_i=\id\) for \(i\neq 1,4\), so that \(\rho_\alpha=s_2s_1s_6\).
By \eqref{eq:completion-length}, we have
\(\length(\sigma_\alpha)=0+(0+2+2+2+3+3)=12\), where the second sum
counts the non-base values \(2,3,7\) below each base value.
Applying \(s_6\), \(s_1\), and \(s_2\) to the values of
\(\sigma_\alpha\) in order, we obtain
\(\rho_\alpha\sigma_\alpha=3\,4\,5\,7\,8\,9\mid1\,2\,6=\cmpl{3,4,5,7,8,9}\),
which sends each \(i\in[6]\) to the top \(b_i+\alpha_i\) of its block.
Its length is
\(\length(\rho_\alpha\sigma_\alpha)=0+(2+2+2+3+3+3)=15=12+3\), again by
\eqref{eq:completion-length}.
\end{exam}

\begin{figure}[ht]
\centering
\begin{tikzpicture}[x=0.72cm,y=0.72cm,>=stealth]
  \fill[gray!20] (0,0) rectangle (3,1);
  \fill[gray!20] (5,0) rectangle (7,1);
  \foreach \pos in {1,...,9} { \draw (\pos-1,0) rectangle (\pos,1); }
  \node at (0.5,0.5) {$\mathbf 1$};
  \node at (1.5,0.5) {$2$};
  \node at (2.5,0.5) {$3$};
  \node at (3.5,0.5) {$\mathbf 4$};
  \node at (4.5,0.5) {$\mathbf 5$};
  \node at (5.5,0.5) {$\mathbf 6$};
  \node at (6.5,0.5) {$7$};
  \node at (7.5,0.5) {$\mathbf 8$};
  \node at (8.5,0.5) {$\mathbf 9$};
  \draw[very thick] (0,0) rectangle (3,1);
  \draw[very thick] (3,0) rectangle (4,1);
  \draw[very thick] (4,0) rectangle (5,1);
  \draw[very thick] (5,0) rectangle (7,1);
  \draw[very thick] (7,0) rectangle (8,1);
  \draw[very thick] (8,0) rectangle (9,1);
  \node at (1.5,-0.45) {\footnotesize $B_1$};
  \node at (3.5,-0.45) {\footnotesize $B_2$};
  \node at (4.5,-0.45) {\footnotesize $B_3$};
  \node at (6,-0.45) {\footnotesize $B_4$};
  \node at (7.5,-0.45) {\footnotesize $B_5$};
  \node at (8.5,-0.45) {\footnotesize $B_6$};
  \foreach \pos/\val in {1/0.5,2/3.5,3/4.5,4/5.5,5/7.5,6/8.5} {
    \node[minimum size=0.5cm,inner sep=1pt] (n\pos) at (\pos*1.45-1.2,2.3) {$\pos$};
    \draw[->] (n\pos) -- (\val,1.03);
  }
  \node[right] at (9.2,1.6) {\footnotesize $i\mapsto b_i=\sigma_\alpha(i)$};
\end{tikzpicture}
\caption{The diagram shows the value blocks of
\(\alpha=(2,0,0,1,0,0)\), which partition \([9]\) into the consecutive
intervals \(B_1,\dots,B_6\) of sizes \(3,1,1,2,1,1\).  An arrow joins
each \(i\in[6]\) to its base value \(b_i=\sigma_\alpha(i)\), which is
set in bold.  The two blocks with more than one element are shaded.}
\label{fig:blocks}
\end{figure}
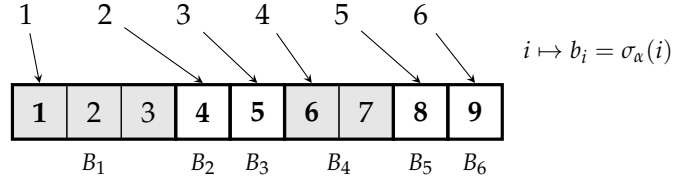

\begin{lem}\label{lem:blocks}
In the notation of \cref{def:blocks}, the following hold.
\begin{enumerate}[label=(\roman*),leftmargin=2.2em]
\item\label{it:partition}
The blocks \(B_1,B_2,\dots,B_n\) are consecutive intervals that
partition \([N]\) in this order, that is, \(b_1=1\),
\(b_{i+1}=b_i+\alpha_i+1\) for \(i<n\), and \(b_n+\alpha_n=N\).  In
particular, \(b_1<b_2<\cdots<b_n\) and
\(b_1+\alpha_1<b_2+\alpha_2<\cdots<b_n+\alpha_n\).
\item\label{it:sigma}
The permutation \(\sigma_\alpha\) is increasing on \([n]\), and
\(\length(\sigma_\alpha)=\sum_{i=1}^n\alpha_i(n-i)\).  For every
polynomial \(f\) in \(x_1,\dots,x_n\) with coefficients in \(A\), we
have
\begin{equation}\label{eq:sigma-substitution}
\sigma_\alpha f=f(x_{b_1},x_{b_2},\dots,x_{b_n}).
\end{equation}
\item\label{it:rho}
The permutation \(\rho_\alpha\) sends each base value \(b_i\) to the
top \(b_i+\alpha_i\) of its block, and
\begin{equation}\label{eq:rho-sigma}
\rho_\alpha\sigma_\alpha=\cmpl{b_1+\alpha_1,\,b_2+\alpha_2,\,\dots,\,b_n+\alpha_n},
\qquad
\length(\rho_\alpha\sigma_\alpha)=\length(\sigma_\alpha)+r.
\end{equation}
\item\label{it:reduced}
The length of \(\rho_\alpha\) is \(r\), the concatenation of the words
\eqref{eq:rho} for \(\rho_1,\dots,\rho_n\) is a reduced word for
\(\rho_\alpha\), and \(\partial_{\rho_\alpha}=\partial_{\rho_1}\cdots\partial_{\rho_n}\)
with \(\partial_{\rho_i}=\partial_{b_i+\alpha_i-1}\cdots\partial_{b_i}\).
\end{enumerate}
\end{lem}

\begin{proof}
Part \ref{it:partition} is immediate from \eqref{eq:blocks}.  For
\ref{it:sigma}, the permutation \(\sigma_\alpha\) is increasing on
\([n]\) because \(b_1<\cdots<b_n\), and thus the
first count in \eqref{eq:completion-length} is zero for
\(\sigma_\alpha\).  The elements of \([N]\setminus\{b_1,\dots,b_n\}\)
smaller than \(b_i\) are the non-base values of
\(B_1,\dots,B_{i-1}\), and there are \(\alpha_1+\cdots+\alpha_{i-1}\)
of them.  Summing over \(i\) gives \(\sum_{j=1}^n\alpha_j(n-j)\).
Formula \eqref{eq:sigma-substitution} follows from \eqref{eq:action},
since \(\sigma_\alpha(i)=b_i\) for \(i\in[n]\).

For \ref{it:rho}, let \(\alpha_i>0\).  When \(\rho_i\) is applied to a
value, its factors act from right to left, so \(s_{b_i}\) moves
\(b_i\) to \(b_i+1\), then \(s_{b_i+1}\) moves it to \(b_i+2\), and so
on up to \(b_i+\alpha_i\).  A value \(b_i+j\) with
\(1\le j\le\alpha_i\) is moved only by \(s_{b_i+j-1}\), which sends it
to \(b_i+j-1\).  Thus \(\rho_i\) fixes every value outside \(B_i\),
sends \(b_i\) to \(b_i+\alpha_i\), and lowers the other values of
\(B_i\) by one.  Since the \(\rho_i\) move disjoint sets of values,
\(\rho_\alpha\) acts in this way on every block at once.  It follows
that \(\rho_\alpha\sigma_\alpha\) sends each position \(i\le n\) to
\(\rho_\alpha(b_i)=b_i+\alpha_i\).  At the positions larger than
\(n\), the permutation \(\sigma_\alpha\) takes the non-base values in
increasing order, and \(\rho_\alpha\) lowers each of them by one.
Since every block is an interval, this gives the values other than the
tops, in increasing order, which proves the first formula in
\eqref{eq:rho-sigma}.  For the length, the first count in
\eqref{eq:completion-length} is zero by \ref{it:partition}.  The
values other than the tops that are smaller than \(b_i+\alpha_i\) are
those of \(B_1,\dots,B_i\), and there are \(\alpha_1+\cdots+\alpha_i\)
of them.  Summing over \(i\) gives
\(\sum_{j=1}^n\alpha_j(n-j)+r=\length(\sigma_\alpha)+r\), which proves
\ref{it:rho}.

For \ref{it:reduced}, the letters of the concatenated word are the
\(s_j\) with \(b_i\le j<b_i+\alpha_i\) for some \(i\in[n]\), so by
\ref{it:partition} no simple transposition occurs twice in it.  A word
in which each simple transposition occurs at most once is reduced.
Thus the concatenated word is a reduced word of length \(r\) for
\(\rho_\alpha\), each of its factors \eqref{eq:rho} is reduced as
well, and
\(\partial_{\rho_\alpha}=\partial_{\rho_1}\cdots\partial_{\rho_n}\).
This completes the proof.
\end{proof}

\begin{prop}\label{prop:extraction}
For every \(f\in A[x_1,\dots,x_n]\), we have
\begin{equation}\label{eq:extraction}
\bigl(\partial_{\rho_\alpha}\sigma_\alpha f\bigr)\big|_{\xx=0}
=[\xx^\alpha]f.
\end{equation}
\end{prop}

\begin{proof}
Since both sides of \eqref{eq:extraction} are \(A\)-linear in \(f\),
it is enough to prove the proposition for a monomial \(f=\xx^\gamma\)
with \(\gamma\in\NN^n\).  By \eqref{eq:sigma-substitution} and
\cref{lem:blocks}\ref{it:reduced}, we have
\(\sigma_\alpha f=\prod_{i=1}^nx_{b_i}^{\gamma_i}\) and
\(\partial_{\rho_\alpha}=\partial_{\rho_1}\cdots\partial_{\rho_n}\).
The chain \(\partial_{\rho_i}=\partial_{b_i+\alpha_i-1}\cdots\partial_{b_i}\)
involves only the variables \(x_j\) with \(j\in B_i\), and
\(x_{b_i}\) is the only one of them that occurs in this monomial.
Since the blocks are disjoint, \(\partial_{\rho_\alpha}\sigma_\alpha f\)
is the product of the polynomials \(\partial_{\rho_i}(x_{b_i}^{\gamma_i})\),
and by \cref{lem:newton}\ref{it:newton-coeff} for the variable
\(x_{b_i}\) and \(k=\alpha_i\), the value of the \(i\)th factor at
\(\xx=0\) is the coefficient of \(x_{b_i}^{\alpha_i}\) in
\(x_{b_i}^{\gamma_i}\).  Thus the left side of \eqref{eq:extraction} is \(1\)
if \(\gamma=\alpha\) and \(0\) otherwise, which is the coefficient of
\(\xx^\alpha\) in \(\xx^\gamma\), as desired.
\end{proof}

\begin{exam}\label{ex:extraction}
With the same setting as in \cref{ex:blocks}, let us compute the left
side of \eqref{eq:extraction} for the polynomial \(F\) of
\eqref{eq:ex-F}.  By \eqref{eq:sigma-substitution}, the block
placement renames \(x_1,\dots,x_6\) as \(x_1\), \(x_4\), \(x_5\),
\(x_6\), \(x_8\), \(x_9\), respectively, so that
\[
\sigma_\alpha F=P\cdot(x_6+x_9-y_1-y_2),
\qquad
\partial_6 \sigma_\alpha F=P,
\]
where \(P=(x_1-y_1)(x_1-y_2)(x_1-y_3)=x_1^3-e_1x_1^2+e_2x_1-e_3\), with
\(e_1,e_2,e_3\) as in \cref{ex:F}.  Applying
\cref{lem:newton}\ref{it:newton-chain} to each power of \(x_1\) in
\(P\), we obtain
\begin{align*}
\partial_1P&=(x_1^2+x_1x_2+x_2^2)-e_1(x_1+x_2)+e_2,\\
\partial_2\partial_1P&=(x_1+x_2+x_3)-e_1.
\end{align*}
Since \(\partial_{\rho_\alpha}=\partial_2\partial_1\partial_6\),
evaluating at \(\xx=0\) gives
\[
\bigl(\partial_2\partial_1\partial_6\sigma_\alpha F\bigr)\big|_{\xx=0}
=-e_1=-y_1-y_2-y_3=[x_1^2x_4]F
\]
by \eqref{eq:ex-coefficients}, as \cref{prop:extraction} asserts.
Every other monomial of \(F\) contributes zero at \(\xx=0\).
\end{exam}

We call \(\partial_{\rho_\alpha}\sigma_\alpha\) the \emph{extraction
operator} of \(\alpha\).  \Cref{prop:extraction} will be used in the
proof of \cref{thm:main} in
\cref{sec:main}, with \(A=\ZZ[\yy]\) and
\(f=\partial_{w/v}\SSch_u(\xx;\yy)\), which is a polynomial in
\(x_1,\dots,x_n\) by \cref{prop:skew-triple}.


\section{Lifting the skew divided difference}
\label{sec:lifting}

In this section we define the block lift (\cref{def:lifts}) and prove
\cref{thm:operator}, which says that
\(\partial_{\widetilde w_\alpha/\widetilde v_\alpha}
=\partial_{\rho_\alpha}\sigma_\alpha\partial_{w/v}\), where
\(\partial_{\rho_\alpha}\sigma_\alpha\) is the extraction operator of
\cref{prop:extraction}.  The proof applies the
factorization rule of \cref{lem:factor} twice, once to
\(\sigma_\alpha\) and once to \(\rho_\alpha\), and only a single term
survives each time.

\subsection{A factorization rule}

\begin{lem}\label{lem:factor}
Let \(u,w\in S_N\) with
\(\length(uw)=\length(u)+\length(w)\).  Then, for
every \(v\in S_N\),
\begin{equation}\label{eq:factor}
\partial_{uw/v}
=\sum_{v_1,v_2}\partial_{u/v_1}\,\partial_{w/v_2},
\end{equation}
where the sum is over all pairs \((v_1,v_2)\in S_N^2\) with
\(v_1v_2=v\) and \(\length(v_1)+\length(v_2)=\length(v)\).
Moreover, a pair contributes a nonzero term only if
\(v_1\le u\) and \(v_2\le w\).
\end{lem}

\begin{proof}
By the uniqueness in \cref{lem:leibniz}, it is sufficient to show that
the operators \(T_v\) given by the right side of
\eqref{eq:factor} satisfy \eqref{eq:leibniz} with \(uw\) in
place of \(w\).  Recall that \(\partial_{v_1}\partial_{v_2}\) is
\(\partial_{v_1v_2}\) if
\(\length(v_1v_2)=\length(v_1)+\length(v_2)\) and is \(0\) otherwise.
In particular, \(\partial_{uw}=\partial_u\partial_w\), and for
\(f,g\in A[\xx]\) we have
\begin{align*}
\partial_{uw}(fg)
&=\partial_u\sum_{v_2\in S_N}
\bigl(\partial_{w/v_2}f\bigr)\bigl(\partial_{v_2}g\bigr)
=\sum_{v_1,v_2\in S_N}
\bigl(\partial_{u/v_1}\partial_{w/v_2}f\bigr)
\bigl(\partial_{v_1}\partial_{v_2}g\bigr)\\
&=\sum_{v\in S_N}\bigl(T_vf\bigr)\bigl(\partial_vg\bigr),
\end{align*}
where the last equality uses the rule for
\(\partial_{v_1}\partial_{v_2}\), with the terms grouped according to
\(v=v_1v_2\).  This proves
\eqref{eq:factor}, and the last assertion follows from
\ref{it:skew-bruhat}.
\end{proof}

Samuel~\cite{Samuel2024} stated the case
\(w=s_i\) of \eqref{eq:factor} as a recurrence for the same
operators.  We have the following corollary describing a skew divided
difference
whose endpoints \(\sigma w\) and \(\sigma v\) share a left factor
\(\sigma\) that is increasing on \([n]\).

\begin{cor}\label{cor:prefix}
For \(n\le N\), let \(\sigma\in S_N\) be increasing on \([n]\), and let
\(v,w\in S_n\subseteq S_N\).  Then
\[
\partial_{\sigma w/\sigma v}=\sigma\partial_{w/v}.
\]
\end{cor}

\begin{proof}
By \eqref{eq:prefix-length}, we can apply \cref{lem:factor} with
\(\sigma\) and \(\sigma v\) in place of \(u\) and \(v\), and
\eqref{eq:factor} expresses
\(\partial_{\sigma w/\sigma v}\) as the sum of the products
\(\partial_{\sigma/v_1}\partial_{w/v_2}\) over the pairs with
\(v_1v_2=\sigma v\) and
\(\length(v_1)+\length(v_2)=\length(\sigma v)\).  In a nonzero term
we have \(v_2\le w\), so \(v_2\in S_n\).  We also have \(v_1\le\sigma\), so that
\(\length(v_1)\le\length(\sigma)\).  By \eqref{eq:prefix-length}, the
permutation \(\sigma\) is the only element of the coset \(\sigma S_n\) of
length at most \(\length(\sigma)\), and
\(v_1=\sigma vv_2^{-1}\) is such an element, so that
\(v_1=\sigma\) and \(v_2=v\).  This pair occurs in the sum,
since \(\length(\sigma)+\length(v)=\length(\sigma v)\) by
\eqref{eq:prefix-length}, and its term is \(\sigma\partial_{w/v}\)
because \(\partial_{\sigma/\sigma}=\sigma\) by \ref{it:skew-ends}, as
desired.
\end{proof}

If \(v\le w\), then \(\sigma v\le\sigma w\) by \eqref{eq:bruhat}, since
by \eqref{eq:prefix-length} a reduced word for \(\sigma\) followed by a
reduced word for \(w\) is a reduced word for \(\sigma w\).

We will also need the following result, in which every letter of a
reduced word for the left factor \(u\) is a left ascent of the
lower endpoint.

\begin{cor}\label{cor:ascent}
Let \(u,v,w\in S_N\) with
\(\length(uw)=\length(u)+\length(w)\), and suppose
that every letter of some reduced word for \(u\) is a left ascent
of \(v\).  Then
\begin{equation}\label{eq:ascent-extension}
\partial_{uw/v}=\partial_u\partial_{w/v}.
\end{equation}
\end{cor}

\begin{proof}
For a contradiction, suppose that a pair \((v_1,v_2)\) with
\(v_1\ne\id\) contributes a nonzero term to \eqref{eq:factor}.  Then
\(v_1\le u\) by \cref{lem:factor}, so some
reduced word \(s_{j_1}s_{j_2}\cdots s_{j_m}\) for \(v_1\) is a
subword of the given reduced word for \(u\).  Its first letter
\(s_{j_1}\) is a left ascent of \(v\) by hypothesis, and
\(\length(s_{j_1}v_1)=\length(v_1)-1\).  Since
\(\length(v_1)+\length(v_2)=\length(v)\), it follows that
\[
\length(s_{j_1}v)=\length(s_{j_1}v_1v_2)
\le\length(s_{j_1}v_1)+\length(v_2)=\length(v)-1,
\]
so \(s_{j_1}\) is a left descent of \(v\), a contradiction.  Thus
only the pair \((\id,v)\) contributes to \eqref{eq:factor}, and
\(\partial_{u/\id}=\partial_u\) by \ref{it:skew-ends}.
This proves \eqref{eq:ascent-extension}.
\end{proof}

\subsection{The block lift}\label{sec:block-lift}

For the remainder of the section, we fix \(u,v,w\in S_n\) and
\(\alpha\in\NN^n\), put \(r=\abs{\alpha}\) and \(N=n+r\), and
let \(b_i\), \(B_i\),
\(\sigma_\alpha\), \(\rho_i\), and \(\rho_\alpha\) be as in
\cref{def:blocks}.

\begin{defn}\label{def:lifts}
The \emph{block lift} of
\((u,v,w,\alpha)\) is the triple
\((\widetilde u_\alpha,\widetilde v_\alpha,\widetilde w_\alpha)\in S_N^3\)
given by
\begin{equation}\label{eq:lifts}
\widetilde u_\alpha=u\times1_r,\qquad
\widetilde v_\alpha=\sigma_\alpha(v\times1_r),\qquad
\widetilde w_\alpha=\rho_\alpha\sigma_\alpha(w\times1_r).
\end{equation}
We also put \(\widetilde w_\alpha^\circ=\sigma_\alpha(w\times1_r)\), so
that \(\widetilde w_\alpha=\rho_\alpha\widetilde w_\alpha^\circ\).
\end{defn}

See \cref{ex:lift} for a concrete example of the block lift, and
\cref{ex:slices} for further values of \(\alpha\).  Using
\eqref{eq:completion-compose}, we can write down an explicit
expression for each of the four permutations of \cref{def:lifts} as an
increasing completion.

\begin{lem}\label{lem:explicit}
In the notation of \cref{def:completion}, we have
\begin{align}
\widetilde u_\alpha&=\cmpl{u(1),\dots,u(n)},
&
\widetilde v_\alpha&=\cmpl{b_{v(1)},\dots,b_{v(n)}},
\label{eq:explicit-uv}\\
\widetilde w_\alpha^\circ&=\cmpl{b_{w(1)},\dots,b_{w(n)}},
&
\widetilde w_\alpha&=\cmpl{b_{w(1)}+\alpha_{w(1)},\dots,b_{w(n)}+\alpha_{w(n)}}.
\label{eq:explicit-w}
\end{align}
Moreover, every \(s_j\) with \(b_i\le j<b_i+\alpha_i\) for some
\(i\in[n]\) is a left ascent of \(\widetilde v_\alpha\).  These \(s_j\)
are the letters of the reduced word of \(\rho_\alpha\) in
\cref{lem:blocks}\ref{it:reduced}.
\end{lem}

\begin{proof}
The one-line notation of \(u\times1_r\) is \(u(1)\cdots u(n)\) followed
by \(n+1,\dots,N\), which gives the formula for \(\widetilde u_\alpha\).
Since \(\sigma_\alpha=\cmpl{b_1,\dots,b_n}\) by \eqref{eq:sigma} and
\(\rho_\alpha\sigma_\alpha=\cmpl{b_1+\alpha_1,\dots,b_n+\alpha_n}\) by
\eqref{eq:rho-sigma}, the other three formulas follow from
\eqref{eq:lifts} and \eqref{eq:completion-compose}.

For the second assertion, the letters in question are the
\(s_{k-1}\) with \(b_i+1\le k\le b_i+\alpha_i\) for some \(i\in[n]\), so
that \(k\) is not a base value.  The first \(n\) entries of
\(\widetilde v_\alpha\) are base values by \eqref{eq:explicit-uv}, and
\eqref{eq:completion-ascent} shows that each such \(s_{k-1}\) is a
left ascent of \(\widetilde v_\alpha\).  This completes the proof.
\end{proof}

By \eqref{eq:explicit-uv} and \eqref{eq:explicit-w}, the permutation
\(\widetilde v_\alpha\) sends each \(i\le n\) to the smallest
element of the block \(B_{v(i)}\), and \(\widetilde w_\alpha^\circ\)
and \(\widetilde w_\alpha\) send it to the smallest and the largest
element of \(B_{w(i)}\), respectively.

\begin{thm}\label{thm:operator}
We have
\begin{equation}\label{eq:operator}
\partial_{\widetilde w_\alpha/\widetilde v_\alpha}
=\partial_{\rho_\alpha}\sigma_\alpha\partial_{w/v}.
\end{equation}
Moreover,
\begin{equation}\label{eq:lengths}
\length(\widetilde v_\alpha)=\length(v)+\length(\sigma_\alpha),
\qquad
\length(\widetilde w_\alpha)=\length(w)+\length(\sigma_\alpha)+r,
\end{equation}
and if \(v\le w\), then \(\widetilde v_\alpha\le\widetilde w_\alpha\).
\end{thm}

\begin{proof}
By \cref{lem:blocks}\ref{it:sigma}, \cref{cor:prefix} applies with
\(\sigma=\sigma_\alpha\).  Since \(\partial_{w/v}\) is the same operator
whether \(v\) and \(w\) are regarded in \(S_n\) or in \(S_N\), we obtain
\begin{equation}\label{eq:base-identity}
\partial_{\widetilde w_\alpha^\circ/\widetilde v_\alpha}=\sigma_\alpha\partial_{w/v},
\end{equation}
and \(\widetilde v_\alpha\le\widetilde w_\alpha^\circ\) if \(v\le w\).

Since \(\sigma_\alpha\) and \(\rho_\alpha\sigma_\alpha\) are increasing on
\([n]\) by \cref{lem:blocks}, \eqref{eq:prefix-length} and
\eqref{eq:rho-sigma} give
\[
\length(\widetilde v_\alpha)=\length(v)+\length(\sigma_\alpha),\qquad
\length(\widetilde w_\alpha^\circ)=\length(w)+\length(\sigma_\alpha),\qquad
\length(\widetilde w_\alpha)=\length(w)+\length(\sigma_\alpha)+r.
\]
This proves \eqref{eq:lengths}, and
\(\length(\widetilde w_\alpha)=\length(\rho_\alpha)+\length(\widetilde w_\alpha^\circ)\)
since \(\length(\rho_\alpha)=r\).  Every letter of the reduced word of
\(\rho_\alpha\) is a left ascent of \(\widetilde v_\alpha\) by
\cref{lem:explicit}, so \cref{cor:ascent} gives
\(\partial_{\widetilde w_\alpha/\widetilde v_\alpha}
=\partial_{\rho_\alpha}\partial_{\widetilde w_\alpha^\circ/\widetilde v_\alpha}\),
and \eqref{eq:operator} follows from \eqref{eq:base-identity}.
Finally, the reduced word of \(\rho_\alpha\) followed by a reduced word
for \(\widetilde w_\alpha^\circ\) is a reduced word for
\(\widetilde w_\alpha\), so
\(\widetilde w_\alpha^\circ\le\widetilde w_\alpha\), and hence
\(\widetilde v_\alpha\le\widetilde w_\alpha\) if \(v\le w\).  This
completes the proof.
\end{proof}

\begin{rmk}
For \(v=w=\id\), the block lift is
\((\widetilde u_\alpha,\sigma_\alpha,\rho_\alpha\sigma_\alpha)\), and
\eqref{eq:operator} says that the extraction operator
\(\partial_{\rho_\alpha}\sigma_\alpha\) of \cref{prop:extraction} is the
skew divided difference \(\partial_{\rho_\alpha\sigma_\alpha/\sigma_\alpha}\).
Thus \eqref{eq:operator} can be written as
\(\partial_{\rho_\alpha\sigma_\alpha/\sigma_\alpha}\partial_{w/v}
=\partial_{\rho_\alpha\sigma_\alpha(w\times1_r)/\sigma_\alpha(v\times1_r)}\).
\end{rmk}

\begin{exam}\label{ex:lift}
Take the same setting as in \cref{ex:blocks}, where
\(\alpha=(2,0,0,1,0,0)\), \(N=9\), and the base values are
\((b_1,\dots,b_6)=(1,4,5,6,8,9)\).  By \cref{def:lifts,lem:explicit},
the block lift and the permutation \(\widetilde w_\alpha^\circ\) are
\begin{align*}
\widetilde u_\alpha&=u\times1_3=4\,3\,5\,1\,2\,6\mid7\,8\,9,\\
\widetilde v_\alpha&=\cmpl{b_{v(1)},\dots,b_{v(6)}}
=\cmpl{b_1,b_4,b_2,b_6,b_3,b_5}
=1\,6\,4\,9\,5\,8\mid2\,3\,7,\\
\widetilde w_\alpha^\circ&=\cmpl{b_{w(1)},\dots,b_{w(6)}}
=\cmpl{b_1,b_6,b_4,b_3,b_2,b_5}
=1\,9\,6\,5\,4\,8\mid2\,3\,7,\\
\widetilde w_\alpha&=\cmpl{b_1+2,\,b_6,\,b_4+1,\,b_3,\,b_2,\,b_5}
=3\,9\,7\,5\,4\,8\mid1\,2\,6.
\end{align*}
As described after \cref{lem:explicit}, each value \(i\) of
\(v=142635\) is replaced in \(\widetilde v_\alpha\) by the smallest
element \(b_i\) of \(B_i\), and each value \(i\) of \(w=164325\) is
replaced in \(\widetilde w_\alpha\) by the largest element
\(b_i+\alpha_i\) of \(B_i\).  \Cref{fig:perms}
gives the dot diagrams of \(\widetilde v_\alpha\),
\(\widetilde w_\alpha^\circ\), and \(\widetilde w_\alpha\).  The
lengths of \(\widetilde v_\alpha\) and \(\widetilde w_\alpha\) are
\[
\length(\widetilde v_\alpha)=16=4+12,\qquad
\length(\widetilde w_\alpha)=22=7+12+3,
\]
which agree with \eqref{eq:lengths}.  For example, in the formula
\eqref{eq:completion-length} for \(\widetilde v_\alpha\), the first
count is \(\length(v)=4\), given by the pairs of entries \((6,4)\),
\((6,5)\), \((9,5)\), and \((9,8)\) of \(1\,6\,4\,9\,5\,8\).  The
numbers of non-base values \(2,3,7\) smaller than these six entries are
\(0,2,2,3,2,3\), which sum to \(\length(\sigma_\alpha)=12\).

By \cref{lem:explicit}, the three letters \(s_1\), \(s_2\), \(s_6\) of
\(\rho_\alpha=s_2s_1s_6\) are left ascents of
\(\widetilde v_\alpha=1\,6\,4\,9\,5\,8\mid2\,3\,7\), and
\eqref{eq:operator} reads
\begin{equation}\label{eq:ex-identity}
\partial_{\widetilde w_\alpha/\widetilde v_\alpha}
=\partial_2\partial_1\partial_6\sigma_\alpha\partial_{w/v},
\end{equation}
where \(\partial_2\partial_1=\partial_{\rho_1}\) and
\(\partial_6=\partial_{\rho_4}\).  The right side of
\eqref{eq:ex-identity} sends \(\SSch_u(\xx;\yy)\) to the polynomial
\(\partial_2\partial_1\partial_6\sigma_\alpha F\) computed in
\cref{ex:extraction}.
\end{exam}

\begin{figure}[ht]
\centering
\begin{tikzpicture}[x=0.34cm,y=0.34cm]
  \fill[gray!20] (0,0) rectangle (3,9);
  \fill[gray!20] (5,0) rectangle (7,9);
  \draw (0,0) rectangle (9,9);
  \draw[thick] (0,6) -- (9,6);
  \foreach \pos in {3,4,5,7,8} { \draw[thick] (\pos,0) -- (\pos,9); }
  \foreach \pos in {1,...,9} { \node[below] at (\pos-0.5,0) {\tiny \pos}; \node[left] at (0,\pos-0.5) {\tiny \pos}; }
  \foreach \pos/\val in {1/1,2/6,3/4,4/9,5/5,6/8,7/2,8/3,9/7} { \fill (\val-0.5,\pos-0.5) circle (0.15); }
  \node[above] at (4.5,9) {$\widetilde v_\alpha$};
\end{tikzpicture}
\qquad\qquad
\begin{tikzpicture}[x=0.34cm,y=0.34cm]
  \fill[gray!20] (0,0) rectangle (3,9);
  \fill[gray!20] (5,0) rectangle (7,9);
  \draw (0,0) rectangle (9,9);
  \draw[thick] (0,6) -- (9,6);
  \foreach \pos in {3,4,5,7,8} { \draw[thick] (\pos,0) -- (\pos,9); }
  \foreach \pos in {1,...,9} { \node[below] at (\pos-0.5,0) {\tiny \pos}; \node[left] at (0,\pos-0.5) {\tiny \pos}; }
  \foreach \pos/\val in {1/1,2/9,3/6,4/5,5/4,6/8,7/2,8/3,9/7} { \fill (\val-0.5,\pos-0.5) circle (0.15); }
  \node[above] at (4.5,9) {$\widetilde w_\alpha^\circ$};
\end{tikzpicture}
\qquad\qquad
\begin{tikzpicture}[x=0.34cm,y=0.34cm]
  \fill[gray!20] (0,0) rectangle (3,9);
  \fill[gray!20] (5,0) rectangle (7,9);
  \draw (0,0) rectangle (9,9);
  \draw[thick] (0,6) -- (9,6);
  \foreach \pos in {3,4,5,7,8} { \draw[thick] (\pos,0) -- (\pos,9); }
  \foreach \pos in {1,...,9} { \node[below] at (\pos-0.5,0) {\tiny \pos}; \node[left] at (0,\pos-0.5) {\tiny \pos}; }
  \foreach \pos/\val in {1/3,2/9,3/7,4/5,5/4,6/8,7/1,8/2,9/6} { \fill (\val-0.5,\pos-0.5) circle (0.15); }
  \node[above] at (4.5,9) {$\widetilde w_\alpha$};
\end{tikzpicture}
\caption{The left, middle, and right diagrams show
\(\widetilde v_\alpha\), \(\widetilde w_\alpha^\circ\), and
\(\widetilde w_\alpha\) for \(\alpha=(2,0,0,1,0,0)\), with a dot in
row \(i\) and column \(j\) when the permutation sends \(i\) to \(j\).
Thick vertical lines divide
the columns into the value blocks, the two blocks with more than one
element are shaded, and a thick horizontal line separates the rows
\(1,\dots,6\) from the appended rows \(7,8,9\) above them.  In the rows
\(1,\dots,6\), each dot lies in the leftmost column of its block in the
left and middle diagrams and in the rightmost column of its block in
the right diagram.}
\label{fig:perms}
\end{figure}
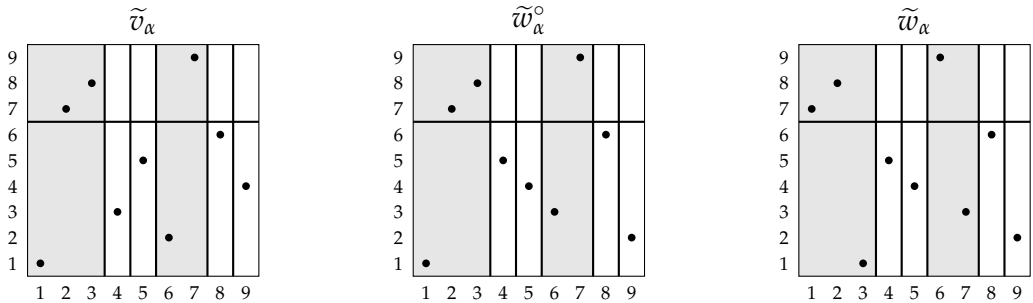


\section{Proof of \texorpdfstring{\cref{thm:main}}{Theorem 1.2} and consequences}
\label{sec:main}

Throughout this section, let \(u,v,w\in S_n\), and let
\(d=\length(u)+\length(v)-\length(w)\).  For \(\alpha\in\NN^n\) we put
\(r=\abs{\alpha}\) and \(N=n+r\), and we denote by
\((\widetilde u_\alpha,\widetilde v_\alpha,\widetilde w_\alpha)\in S_N^3\) the block
lift of \((u,v,w,\alpha)\) (\cref{def:lifts}).  Since
\(\SSch_u(\xx;\yy)\) is homogeneous of degree \(\length(u)\)
and \(\partial_{w/v}\) lowers degrees by
\(\length(w)-\length(v)\), the polynomial
\(\partial_{w/v}\SSch_u(\xx;\yy)\) is homogeneous of total degree \(d\)
in the variables \(x_i\) and \(y_j\), or zero.  In particular, it is
zero when \(d<0\).

\subsection{Proof of \texorpdfstring{\cref{thm:main}}{Theorem 1.2}}

\begin{proof}
Fix \(\alpha\in\NN^n\).  By \cref{prop:skew-triple}, the polynomial
\(\partial_{w/v}\SSch_u(\xx;\yy)\) lies in \(\ZZ[\yy][x_1,\dots,x_n]\),
and we have
\begin{align*}
c^{\widetilde w_\alpha}_{\widetilde u_\alpha,\widetilde v_\alpha}(\yy,0)
&=\bigl(\partial_{\widetilde w_\alpha/\widetilde v_\alpha}
\SSch_{\widetilde u_\alpha}(\xx;\yy)\bigr)\big|_{\xx=0}
=\bigl(\partial_{\widetilde w_\alpha/\widetilde v_\alpha}
\SSch_{u}(\xx;\yy)\bigr)\big|_{\xx=0}\\
&=\bigl(\partial_{\rho_\alpha}\sigma_\alpha\partial_{w/v}
\SSch_{u}(\xx;\yy)\bigr)\big|_{\xx=0}
=[\xx^\alpha]\,\partial_{w/v}\SSch_u(\xx;\yy),
\end{align*}
where the four equalities are \cref{prop:skew-triple} for the block
lift in \(S_N\) at \(\xx=0\), the stability \ref{S:stable},
\cref{thm:operator} with \(A=\ZZ[\yy]\), and \cref{prop:extraction}
with \(A=\ZZ[\yy]\).
Since \(\partial_{w/v}\SSch_u(\xx;\yy)\) is a polynomial in
\(x_1,\dots,x_n\) over \(\ZZ[\yy]\), it is the sum of the terms
\(\bigl([\xx^\alpha]\,\partial_{w/v}\SSch_u(\xx;\yy)\bigr)\,\xx^\alpha\)
over \(\alpha\in\NN^n\).  The terms with \(\abs{\alpha}>d\) vanish,
because this polynomial is homogeneous of degree \(d\) or zero.
Substituting the chain above into the remaining terms, we obtain the
desired identity \eqref{eq:main}.
\end{proof}

Renaming the alphabet \(\xx\) as \(\tb\) in \eqref{eq:main} and
applying \eqref{eq:skew-triple} to the left side, we obtain the
equivalent identity
\begin{equation}\label{eq:main-triple}
c^{w}_{u,v}(\yy,\tb)
=\sum_{\substack{\alpha\in\NN^n\\ \abs{\alpha}\le d}}
c^{\widetilde w_\alpha}_{\widetilde u_\alpha,\widetilde v_\alpha}(\yy,0)\,\tb^\alpha.
\end{equation}

The block lift depends on \(n\) and not only on \(u\), \(v\), \(w\),
and the monomial \(\xx^\alpha\), since it is not stable under the
inclusion of \(S_n\) in \(S_{n+1}\).  If \(n\) is replaced by \(n+1\)
and \(\alpha\) by \((\alpha,0)\in\NN^{n+1}\), which gives the same
monomial, then the block placement
\(\sigma_{(\alpha,0)}\in S_{N+1}\) sends the position \(n+1\) to the new
base value \(N+1\).  Thus for \(r>0\) the new block lift differs from
\((\widetilde u_\alpha\times1_1,\widetilde v_\alpha\times1_1,\widetilde w_\alpha\times1_1)\),
although \eqref{eq:main} holds for both \(n\) and \(n+1\).

\subsection{Consequences}
\label{sec:consequences}

Setting \(\yy=0\) in \eqref{eq:main}, we obtain \cref{conj:kirillov} in the
following form.

\begin{cor}\label{cor:top}
We have
\[
\partial_{w/v}\SSch_u(\xx)
=\sum_{\substack{\alpha\in\NN^n\\ \abs{\alpha}=d}}
c^{\widetilde w_\alpha}_{\widetilde u_\alpha,\widetilde v_\alpha}\,\xx^\alpha.
\]
In particular, \(\partial_{w/v}\SSch_u(\xx)\) has nonnegative integer
coefficients.
\end{cor}

\begin{proof}
We set \(\yy=0\) in \eqref{eq:main}.  The left side becomes
\(\partial_{w/v}\SSch_u(\xx)\), since this substitution commutes with
\(\partial_{w/v}\) by \ref{it:skew-degree}.
On the right side, each coefficient becomes
\(c^{\widetilde w_\alpha}_{\widetilde u_\alpha,\widetilde v_\alpha}(0,0)
=c^{\widetilde w_\alpha}_{\widetilde u_\alpha,\widetilde v_\alpha}\).
Since \(\partial_{w/v}\SSch_u(\xx)\) is
homogeneous of degree \(d\) or zero, the terms with \(\abs{\alpha}<d\)
vanish.  Finally, the structure constants are nonnegative integers,
which proves the last assertion.
\end{proof}

We now return to the two special cases of \cref{conj:kirillov}
discussed in \cref{sec:intro}.  The first is the case \(d=0\) of the
following corollary, which concerns the term \(\alpha=0\) of
\eqref{eq:main}.

\begin{cor}\label{cor:zero}
The block lift of \((u,v,w,0)\) is \((u,v,w)\), and
\begin{equation}\label{eq:zero-slice}
\partial_{w/v}\SSch_u(\xx;\yy)\big|_{\xx=0}=c^{w}_{u,v}(\yy,0).
\end{equation}
In particular,
\(\partial_{w/v}\SSch_u(\xx)\big|_{\xx=0}=c^{w}_{u,v}\), and if
\(d=0\), then \(\partial_{w/v}\SSch_u(\xx;\yy)=c^{w}_{u,v}\).
\end{cor}

\begin{proof}
For \(\alpha=0\) we have \(r=0\), \(N=n\), and \(b_i=i\) for all
\(i\), so that \(\sigma_\alpha=\id\) and \(\rho_\alpha=\id\), and
\eqref{eq:lifts} gives the block lift \((u,v,w)\).  Identity
\eqref{eq:zero-slice} is \eqref{eq:skew-triple} at \(\xx=0\), and it
is also the term \(\alpha=0\).  As in the proof of
\cref{cor:top}, setting \(\yy=0\) commutes with \(\partial_{w/v}\),
and \(c^{w}_{u,v}(0,0)=c^{w}_{u,v}\), so \eqref{eq:zero-slice} at \(\yy=0\) gives the
second identity.  If \(d=0\), then \(\partial_{w/v}\SSch_u(\xx;\yy)\)
is a constant by homogeneity, and this constant is its value at
\(\xx=0\) and \(\yy=0\), which is \(c^{w}_{u,v}\), as desired.
\end{proof}

The identity
\(\partial_{w/v}\SSch_u(\xx)\big|_{\xx=0}=c^{w}_{u,v}\) of
\cref{cor:zero} was stated by Kirillov~\cite[(1.3)]{Kirillov2007}.

The second special case where \( v=\id \) concerns Schubert polynomials, since
\(\partial_{w/\id}\SSch_u(\xx)\) is the Schubert polynomial
\(\SSch_{uw^{-1}}(\xx)\) or zero.  For \(w\ne\id\), the lower index
\(\widetilde u_\alpha\) in \cref{thm:main} lifts \(u\) rather than
\(uw^{-1}\), and so the coefficients of \(\SSch_u(\xx)\) itself
correspond to the case \(v=w=\id\).  More generally, if \(v=w\), then
\(\partial_{v/v}=v\) by \ref{it:skew-ends}, so the coefficients of
\(\partial_{v/v}\SSch_u(\xx)\) are those of \(\SSch_u(\xx)\) with the
variables permuted by \(v\), and \cref{thm:main} gives the following.

\begin{cor}\label{cor:vw}
Suppose that \(w=v\). Let \(\alpha\in\NN^n\), and define
\(\gamma\in\NN^n\) by \(\gamma_i=\alpha_{v(i)}\) for \(i\in[n]\).  Then
\begin{equation}\label{eq:vw}
[\xx^\alpha]\,v\left(\SSch_u(\xx;\yy)\right)
=[\xx^\gamma]\,\SSch_u(\xx;\yy)
=c^{\widetilde w_\alpha}_{\widetilde u_\alpha,\widetilde v_\alpha}(\yy,0).
\end{equation}
Moreover, if \(\abs{\alpha}=\length(u)\), then
\([\xx^\gamma]\,\SSch_u(\xx)=c^{\widetilde w_\alpha}_{\widetilde u_\alpha,\widetilde v_\alpha}\).
\end{cor}

\begin{proof}
By \ref{it:skew-ends}, \(\partial_{v/v}=v\), so the chain of equalities
in the proof of \cref{thm:main} with \(w=v\) gives the equality of the
outer terms of \eqref{eq:vw}.  By \eqref{eq:action}, the permutation
\(v\) sends \(\xx^\gamma\) to
\(\prod_ix_{v(i)}^{\gamma_i}=\xx^\alpha\), so
\([\xx^\alpha]\,v(f)=[\xx^\gamma]f\) for every polynomial \(f\), which
gives the first equality.  If \(\abs{\alpha}=\length(u)\), then
\(\abs{\alpha}=d\), and \cref{cor:top} with \(w=v\) gives the last
assertion.
\end{proof}

For \(v=\id\) we have \(\gamma=\alpha\), and every monomial of
\(\SSch_u(\xx)\) has degree \(\length(u)\), so the last assertion of
\cref{cor:vw} realizes every coefficient of \(\SSch_u(\xx)\) as a single
structure constant.  We compare this realization with that of Lenart,
Robinson, and Sottile in \cref{sec:lrs}.

For \(\beta\in\NN^n\), we write \([\xx^\alpha\yy^\beta]G\) for the
coefficient of \(\xx^\alpha\yy^\beta\) in a polynomial \(G\).  By the
Cauchy formula~\cite[Chapter~VI, (6.3)]{Macdonald1991}, we have
\[
\SSch_u(\xx;\yy)=\sum_{u_1,u_2}\SSch_{u_1}(\xx)\,\SSch_{u_2}(-\yy),
\]
summed over \(u_1,u_2\in S_n\) with \(u=u_2^{-1}u_1\) and
\(\length(u)=\length(u_1)+\length(u_2)\).  Since \(\partial_{w/v}\) acts
only on \(\xx\), it follows that
\[
(-1)^{\abs{\beta}}\,[\xx^\alpha\yy^\beta]\,\partial_{w/v}\SSch_u(\xx;\yy)
=\sum_{u_1,u_2}\bigl([\xx^\alpha]\,\partial_{w/v}\SSch_{u_1}(\xx)\bigr)
\bigl([\yy^\beta]\,\SSch_{u_2}(\yy)\bigr).
\]
By \cref{cor:top,cor:vw}, each factor on the right is zero or a single
structure constant.  Thus this integer is a sum of products of
structure constants, and in particular it is nonnegative.  In fact,
since \(\widetilde v_\alpha\) and \(\widetilde w_\alpha\) do not depend
on \(u\), comparing this expansion with the Cauchy formula for
\(\SSch_{\widetilde u_\alpha}(\xx;\yy)\) shows that \cref{thm:main} also
follows from \cref{cor:top} applied to every \(u_1\).  Gao and
Xiong~\cite[Theorem~1.1]{GX2025} proved more strongly that
\(c^{w}_{u,v}(\yy,\tb)\) is a polynomial in the differences
\(t_i-y_j\) with nonnegative integer coefficients.  Chen, Fan, Xiong,
and Yao~\cite[Theorem~2.2]{CFXY2026} used this theorem and the Cauchy
formula to show that the nonvanishing of the triple Schubert
coefficients is determined by that of the ordinary structure
constants.

By \cref{thm:main}, on the other hand, the same integer is
\((-1)^{\abs{\beta}}\) times the
coefficient of \(\yy^\beta\) in the single triple Schubert coefficient
\(c^{\widetilde w_\alpha}_{\widetilde u_\alpha,\widetilde v_\alpha}(\yy,0)\).
This suggests the following question.

\begin{ques}\label{ques:signed}
Let \(\alpha,\beta\in\NN^n\).  Can the nonnegative integer
\[
(-1)^{\abs{\beta}}\,[\xx^\alpha\yy^\beta]\,\partial_{w/v}\SSch_u(\xx;\yy)
\]
be described as a single ordinary Schubert structure constant,
indexed by an explicit construction from \((u,v,w,\alpha,\beta)\)?
\end{ques}

\subsection{Comparison with the work of Lenart--Robinson--Sottile}
\label{sec:lrs}

The general substitution formula of Lenart, Robinson, and
Sottile~\cite[Theorem~9.1 and Remark~9.8]{LRS2006} distributes the
variables of a Schubert polynomial among several alphabets and expands
the result in
products of Schubert polynomials in these alphabets, with single
structure constants as coefficients.  When each alphabet receives a
single variable, their formula gives the following in our notation.
Let \(u,v\in S_n\) and \(\mu\in\NN^n\), let \(b_1,\dots,b_n\) be the base
values of \(\mu\), and put \(r=\abs{\mu}\), so that the increasing
completions below lie in \(S_{n+r}\).  Suppose that for each \(i\), the
polynomial \(\SSch_u(\xx)\) has degree at most \(\mu_i\) in
\(x_{v^{-1}(i)}\).  Then for every \(\gamma\in\NN^n\) with
\(\gamma_{v^{-1}(i)}\le\mu_i\) for all \(i\), we have
\begin{equation}\label{eq:lrs}
[\xx^\gamma]\,\SSch_u(\xx)
=c^{\cmpl{b_{v(1)}+\gamma_1,\,\dots,\,b_{v(n)}+\gamma_n}}
_{u\times1_r,\,\cmpl{b_{v(1)},\,\dots,\,b_{v(n)}}}.
\end{equation}

The case \(v=\id\) and \(\mu=(n-1,\dots,1,0)\) is their
Corollary~9.10.  In this case the degree condition holds for every
\(u\in S_n\), and the lower index \(\cmpl{b_1,\dots,b_n}\) is a
Grassmannian permutation that is the same for all
monomials~\cite[Lemma~9.6]{LRS2006}.  Earlier, Kirillov and
Maeno~\cite[(20)]{KM2000} and Bergeron and
Sottile~\cite[Corollary~5.5]{BS2002} showed that each coefficient of
\(\SSch_u(\xx)\) is an intersection number of a Schubert class with a
product of special Schubert classes.

The degree condition is needed for the identity of polynomials in their
Theorem~9.1, but not for \eqref{eq:lrs} itself.  Their formula comes
from an embedding of \(\mathrm{Fl}_{\mu_1+1}\times\cdots\times\mathrm{Fl}_{\mu_n+1}\)
into \(\mathrm{Fl}_{n+r}\), where \(\mathrm{Fl}_m\) is the variety of
complete flags in \(\mathbb C^m\).  Adeyemo and
Sottile~\cite[Theorems~2.5 and~2.7]{AS2014} showed that the pullback
along this embedding sends \(x_{v^{-1}(i)}\) to the first Chern root of
the \(i\)th factor, and that it expands the Schubert class of
\(u\times1_r\) in the Schubert basis of the product with coefficients
as on the right side of \eqref{eq:lrs}.  The
powers of this root with exponents at most \(\mu_i\) are distinct
Schubert classes, and its \((\mu_i+1)\)st power is zero.  Comparing
coefficients, we obtain \eqref{eq:lrs} for all \(u\in S_n\) and
\(\mu\in\NN^n\), without the degree condition.

By \cref{lem:explicit}, the last assertion of \cref{cor:vw} is
\eqref{eq:lrs} for \(\mu=\alpha\), the smallest \(\mu\) allowed by
\(\gamma\).  For this choice the degree condition fails in general, so
the argument of the previous paragraph is needed.  With the Cauchy
formula, as in \cref{sec:consequences}, this also gives
\eqref{eq:vw}.  Thus \cref{cor:vw} can be deduced from these results,
whereas our proof uses only divided differences.  Compared with their
Corollary~9.10, our realization uses permutations in
\(S_{n+\length(u)}\) rather than in \(S_{n(n+1)/2}\), and its lower
index depends on the monomial.

Suppose now that \(v<w\) and \(\partial_{w/v}\SSch_u(\xx)\ne0\).  This
polynomial is homogeneous of degree
\(\length(u)+\length(v)-\length(w)<\length(u)\), so it is not obtained
from \(\SSch_u(\xx)\) by a substitution of variables, and
\eqref{eq:lrs} for \(u\) does not apply to it.  Nevertheless,
\cref{cor:top} realizes its coefficients with the lower index
\(\widetilde u_\alpha=u\times1_r\).  For \(v=\id\), for example, the
polynomial \(\partial_w\SSch_u(\xx)\) is \(\SSch_{uw^{-1}}(\xx)\), and
\eqref{eq:lrs} realizes its coefficients with \(uw^{-1}\) in place of
\(u\).

\subsection{The running example completed}

\begin{exam}\label{ex:slices}
Take the same setting as in \cref{ex:F}, where \(u=435126\),
\(v=142635\), \(w=164325\), \(d=4\), and
\[
F=\partial_{w/v}\SSch_u(\xx;\yy)
=(x_1^3-e_1x_1^2+e_2x_1-e_3)(x_4+x_6-y_1-y_2).
\]
For \(\alpha=(2,0,0,1,0,0)\), whose block lift is computed in
\cref{ex:lift}, \cref{thm:main} asserts that
\[
[x_1^2x_4]\,F=c^{397548126}_{435126789,\,164958237}(\yy,0),
\]
and both sides are equal to \(-y_1-y_2-y_3\).  The left side was read
off from \(F\) in \cref{ex:F}, and \cref{ex:extraction} recovered it
with the operator \(\partial_2\partial_1\partial_6\sigma_\alpha\) of
\eqref{eq:ex-identity}.  By \eqref{eq:main-triple}, it is also the
coefficient of \(t_1^2t_4\) in \eqref{eq:ex-triple}.
\Cref{tab:slices} lists six of the twelve exponents \(\alpha\) of the
monomials of \(F\) with the permutations \(\widetilde v_\alpha\) and
\(\widetilde w_\alpha\) of their block lifts.  For \(\alpha=0\) the block
lift is \((u,v,w)\), and the constant term \(e_3(y_1+y_2)\) of \(F\) is
\(c^{w}_{u,v}(\yy,0)\), as in \cref{cor:zero}.  For the two exponents
with \(\abs{\alpha}=d=4\), the values are the structure constants
\(c^{\widetilde w_\alpha}_{\widetilde u_\alpha,\widetilde v_\alpha}=1\) in
\(S_{10}\), the two coefficients of
\(\partial_{w/v}\SSch_u(\xx)=x_1^3x_4+x_1^3x_6\) (\cref{cor:top}).
Summing over all twelve exponents, \eqref{eq:main-triple} recovers
\eqref{eq:ex-triple} as
\[
c^{164325}_{435126,\,142635}(\yy,\tb)
=\sum_{\alpha}c^{\widetilde w_\alpha}_{\widetilde u_\alpha,\widetilde v_\alpha}(\yy,0)\,\tb^\alpha
=(t_1-y_1)(t_1-y_2)(t_1-y_3)(t_4+t_6-y_1-y_2).
\]
\end{exam}

\begin{table}[ht]
\centering
\small
\begin{tabular}{cllllc}
\hline
\(\alpha\) & \(\widetilde v_\alpha\) & \(\widetilde w_\alpha\) &
\(\length(\widetilde v_\alpha)\) & \(\length(\widetilde w_\alpha)\) & \([\xx^\alpha]F=c^{\widetilde w_\alpha}_{\widetilde u_\alpha,\widetilde v_\alpha}(\yy,0)\)\\
\hline
\((0,0,0,0,0,0)\) & \(1\,4\,2\,6\,3\,5\) & \(1\,6\,4\,3\,2\,5\) & \(4\) & \(7\) & \(y_1^2y_2y_3+y_1y_2^2y_3\)\\
\((1,0,0,1,0,0)\) & \(1\,5\,3\,8\,4\,7\mid2\,6\) & \(2\,8\,6\,4\,3\,7\mid1\,5\) & \(11\) & \(16\) & \(y_1y_2+y_1y_3+y_2y_3\)\\
\((2,0,0,1,0,0)\) & \(1\,6\,4\,9\,5\,8\mid2\,3\,7\) & \(3\,9\,7\,5\,4\,8\mid1\,2\,6\) & \(16\) & \(22\) & \(-y_1-y_2-y_3\)\\
\((2,0,0,0,0,1)\) & \(1\,6\,4\,8\,5\,7\mid2\,3\,9\) & \(3\,9\,6\,5\,4\,7\mid1\,2\,8\) & \(14\) & \(20\) & \(-y_1-y_2-y_3\)\\
\((3,0,0,1,0,0)\) & \(1\;7\;5\;10\;6\;9\mid2\;3\;4\;8\) & \(4\;10\;8\;6\;5\;9\mid1\;2\;3\;7\) & \(21\) & \(28\) & \(1\)\\
\((3,0,0,0,0,1)\) & \(1\;7\;5\;9\;6\;8\mid2\;3\;4\;10\) & \(4\;10\;7\;6\;5\;8\mid1\;2\;3\;9\) & \(19\) & \(26\) & \(1\)\\
\hline
\end{tabular}
\caption{For six weak compositions \(\alpha\), the table gives the
lifted permutations \(\widetilde v_\alpha\) and \(\widetilde w_\alpha\)
in \(S_N\) with \(N=6+\abs{\alpha}\), their lengths, and the common
value of \([\xx^\alpha]F\) and
\(c^{\widetilde w_\alpha}_{\widetilde u_\alpha,\widetilde v_\alpha}(\yy,0)\).}
\label{tab:slices}
\end{table}

\section*{}

\subsection*{Use of AI}
The author used GPT 5.6 Sol through Codex in the course of this work,
mainly to work out the author's ideas in detail, to verify results by
computation, and to find connections with the literature.  The author
also used Claude Opus 5.5 in preparing the manuscript.  Starting from an
earlier version written by the author, it checked the grammar, polished
the wording, and checked the citations.  The author takes full
responsibility for the contents of the paper.

\subsection*{Acknowledgments}
The author was supported by a KIAS Individual Grant (AP098201) via the Center for Artificial Intelligence and Natural Sciences at Korea Institute for Advanced Study.


\begin{thebibliography}{CFXY26}

\bibitem[AS14]{AS2014}
Praise Adeyemo and Frank Sottile.
\newblock Cohomological consequences of the pattern map.
\newblock {\em Illinois J. Math.}, 58(4):997--1008, 2014.
\newblock \url{https://doi.org/10.1215/ijm/1446819296}.

\bibitem[BB05]{BB2005}
Anders Bj{\"o}rner and Francesco Brenti.
\newblock {\em Combinatorics of {C}oxeter Groups}, volume 231 of {\em Graduate Texts in Mathematics}.
\newblock Springer, New York, 2005.
\newblock \url{https://doi.org/10.1007/3-540-27596-7}.

\bibitem[BGG73]{BGG1973}
I.~N. Bernstein, I.~M. Gel'fand, and S.~I. Gel'fand.
\newblock Schubert cells and cohomology of the spaces {$G/P$}.
\newblock {\em Russian Math. Surveys}, 28(3):1--26, 1973.
\newblock \url{https://doi.org/10.1070/RM1973v028n03ABEH001557}.

\bibitem[BJS93]{BJS1993}
Sara~C. Billey, William Jockusch, and Richard~P. Stanley.
\newblock Some combinatorial properties of {S}chubert polynomials.
\newblock {\em J. Algebraic Combin.}, 2(4):345--374, 1993.
\newblock \url{https://doi.org/10.1023/A:1022419800503}.

\bibitem[BS02]{BS2002}
Nantel Bergeron and Frank Sottile.
\newblock Skew {S}chubert functions and the {P}ieri formula for flag manifolds.
\newblock {\em Trans. Amer. Math. Soc.}, 354(2):651--673, 2002.
\newblock \url{https://doi.org/10.1090/S0002-9947-01-02845-8}.

\bibitem[CFXY26]{CFXY2026}
Yiming Chen, Neil J.~Y. Fan, Rui Xiong, and Ming Yao.
\newblock Non-vanishing of single, double, and triple {S}chubert structure constants.
\newblock Preprint, 2026.
\newblock \url{https://arxiv.org/abs/2608.17378v1}.

\bibitem[Dem74]{Demazure1974}
Michel Demazure.
\newblock D{\'e}singularisation des vari{\'e}t{\'e}s de {S}chubert g{\'e}n{\'e}ralis{\'e}es.
\newblock {\em Ann. Sci. {\'E}cole Norm. Sup. (4)}, 7(1):53--88, 1974.
\newblock \url{https://doi.org/10.24033/asens.1261}.

\bibitem[FGX25]{FGX2025}
Neil J.~Y. Fan, Peter~L. Guo, and Rui Xiong.
\newblock Bumpless pipe dreams meet puzzles.
\newblock {\em Adv. Math.}, 463, 2025.
\newblock Paper No. 110113, \url{https://doi.org/10.1016/j.aim.2025.110113}.

\bibitem[FS94]{FS1994}
Sergey Fomin and Richard~P. Stanley.
\newblock Schubert polynomials and the nil{C}oxeter algebra.
\newblock {\em Adv. Math.}, 103(2):196--207, 1994.
\newblock \url{https://doi.org/10.1006/aima.1994.1009}.

\bibitem[Gra01]{Graham2001}
William Graham.
\newblock Positivity in equivariant {S}chubert calculus.
\newblock {\em Duke Math. J.}, 109(3):599--614, 2001.
\newblock \url{https://doi.org/10.1215/S0012-7094-01-10935-6}.

\bibitem[GX25]{GX2025}
Yibo Gao and Rui Xiong.
\newblock Graham positivity of triple {S}chubert calculus.
\newblock Preprint, 2025.
\newblock \url{https://arxiv.org/abs/2506.09421v1}.

\bibitem[Kir07]{Kirillov2007}
Anatol~N. Kirillov.
\newblock Skew divided difference operators and {S}chubert polynomials.
\newblock {\em SIGMA Symmetry Integrability Geom. Methods Appl.}, 3, 2007.
\newblock Paper No. 072, \url{https://doi.org/10.3842/SIGMA.2007.072}.

\bibitem[KM00]{KM2000}
Anatol~N. Kirillov and Toshiaki Maeno.
\newblock Quantum double {S}chubert polynomials, quantum {S}chubert polynomials and {V}afa-{I}ntriligator formula.
\newblock {\em Discrete Math.}, 217(1-3):191--223, 2000.
\newblock \url{https://doi.org/10.1016/S0012-365X(99)00263-0}.

\bibitem[KT03]{KT2003}
Allen Knutson and Terence Tao.
\newblock Puzzles and (equivariant) cohomology of {G}rassmannians.
\newblock {\em Duke Math. J.}, 119(2):221--260, 2003.
\newblock \url{https://doi.org/10.1215/S0012-7094-03-11922-5}.

\bibitem[Las82]{Lascoux1982}
Alain Lascoux.
\newblock Classes de {C}hern des vari{\'e}t{\'e}s de drapeaux.
\newblock {\em C. R. Acad. Sci. Paris S{\'e}r. I Math.}, 295:393--398, 1982.

\bibitem[Liu15]{Liu2015}
Ricky~Ini Liu.
\newblock Positive expressions for skew divided difference operators.
\newblock {\em J. Algebraic Combin.}, 42(3):861--874, 2015.
\newblock \url{https://doi.org/10.1007/s10801-015-0606-1}.

\bibitem[LRS06]{LRS2006}
Cristian Lenart, Shawn Robinson, and Frank Sottile.
\newblock Grothendieck polynomials via permutation patterns and chains in the {B}ruhat order.
\newblock {\em Amer. J. Math.}, 128(4):805--848, 2006.
\newblock \url{https://doi.org/10.1353/ajm.2006.0034}.

\bibitem[LS82]{LS1982}
Alain Lascoux and Marcel-Paul Sch{\"u}tzenberger.
\newblock Polyn{\^o}mes de {S}chubert.
\newblock {\em C. R. Acad. Sci. Paris S{\'e}r. I Math.}, 294:447--450, 1982.

\bibitem[Mac91]{Macdonald1991}
Ian~G. Macdonald.
\newblock {\em Notes on {S}chubert Polynomials}, volume~6 of {\em Publications du LaCIM}.
\newblock Universit\'e du Qu\'ebec \`a Montr\'eal, Montr\'eal, 1991.

\bibitem[MS99]{MS1999}
Alexander~I. Molev and Bruce~E. Sagan.
\newblock A {L}ittlewood-{R}ichardson rule for factorial {S}chur functions.
\newblock {\em Trans. Amer. Math. Soc.}, 351(11):4429--4443, 1999.
\newblock \url{https://doi.org/10.1090/S0002-9947-99-02381-8}.

\bibitem[Sam24]{Samuel2024}
Matthew~J. Samuel.
\newblock A {M}olev-{S}agan type formula for double {S}chubert polynomials.
\newblock {\em J. Pure Appl. Algebra}, 228(7), 2024.
\newblock Paper No. 107636, \url{https://doi.org/10.1016/j.jpaa.2024.107636}.

\end{thebibliography}
\end{document}